\documentclass[a4paper,10pt]{amsart}
\usepackage[]{amsmath,amsfonts,amssymb,amsthm,amscd,amsxtra,amstext,latexsym,syntonly,tabularx}
\usepackage{graphicx}
\usepackage{mathrsfs}
\usepackage[all,dvips]{xy}
\usepackage[pagebackref]{hyperref}

\theoremstyle{plain}
\newtheorem{xx}{xx}[section]
\newtheorem{thm}[xx]{Theorem}
\newtheorem*{thm*}{Theorem}
\newtheorem{prop}[xx]{Proposition}
\newtheorem{cor}[xx]{Corollary}
\newtheorem{lem}[xx]{Lemma}
\newtheorem*{lem*}{Lemma}

\theoremstyle{definition}
\newtheorem{defn}[xx]{Definition}
\newtheorem*{defn*}{Definition}
\newtheorem{ex}[xx]{Example}

\theoremstyle{remark}
\newtheorem{rem}[xx]{Remark}

\numberwithin{equation}{xx}

\DeclareMathOperator{\coker}{coker}
\DeclareMathOperator{\Def}{Def}
\DeclareMathOperator{\cDef}{\mathsf{Def}}

\DeclareMathOperator{\depth}{depth}

\DeclareMathOperator{\cE}{E}

\DeclareMathOperator{\End}{End}

\DeclareMathOperator{\ev}{ev}
\DeclareMathOperator{\Ext}{Ext}

\DeclareMathOperator{\grade}{grade}

\DeclareMathOperator{\cH}{H}

\DeclareMathOperator{\Hom}{Hom}
\DeclareMathOperator{\uHom}{\ul{\Hom}}

\DeclareMathOperator{\id}{id}
\DeclareMathOperator{\im}{im}

\DeclareMathOperator{\mSpec}{\fr{m}-Spec}

\DeclareMathOperator{\pdim}{pdim}

\DeclareMathOperator{\Proj}{Proj}

\DeclareMathOperator{\Resdim}{res{.}dim}
\DeclareMathOperator{\Rk}{rk}
\DeclareMathOperator{\Sets}{\cat{Sets}}

\DeclareMathOperator{\Soc}{Soc}
\DeclareMathOperator{\Spec}{Spec}

\DeclareMathOperator{\Sym}{Sym}
\DeclareMathOperator{\Syz}{Syz}
\DeclareMathOperator{\Tor}{Tor}

\newcommand{\co}{\colon}
\newcommand{\ra}{\rightarrow}
\newcommand{\la}{\leftarrow}
\newcommand{\lra}{\longrightarrow}

\newcommand{\thr}{\twoheadrightarrow}

\newcommand{\hra}{\hookrightarrow}

\newcommand{\ot}{{\otimes}}
\newcommand{\vare}{\varepsilon}

\newcommand{\vL}{\varLambda}

\newcommand{\lRa}{\Leftrightarrow}

\newcommand{\Ra}{\Rightarrow}

\newcommand{\sbeq}{\subseteq}

\newcommand{\tn}{\textnormal}

\newcommand{\Algf}{\cat{Alg}{}^{\textnormal{fl}}}
\newcommand{\ucoh}{\underline{\cat{coh}}}

\newcommand{\Pf}{\hat{\cat{P}}{}^{\textnormal{fl}}}

\newcommand{\modf}{\cat{mod}{}^{\textnormal{fl}}}

\newcommand{\hot}{{\tilde{\otimes}}}
\newcommand{\resdim}[2]{{#1}\textnormal{-}\Resdim{#2}}

\newcommand\bdot{\ensuremath{%
  \mathchoice%
   {\mskip\thinmuskip\lower0.2ex\hbox{\scalebox{1.5}{$\cdot$}}\mskip\thinmuskip}}%
   {\mskip\thinmuskip\lower0.2ex\hbox{\scalebox{1.5}{$\cdot$}}\mskip\thinmuskip}%
   {\lower0.3ex\hbox{\scalebox{1.2}{$\cdot$}}}%
   {\lower0.3ex\hbox{\scalebox{1.2}{$\cdot$}}}%
}

\renewcommand{\phi}{\varphi}
\renewcommand{\geq}{\geqslant}
\renewcommand{\leq}{\leqslant}

\newcommand{\BB}[1]{\mathbb{{#1}}}

\newcommand{\cat}[1]{\mathsf{{#1}}}

\newcommand{\df}[2]{{\Def}_{#2}^{#1}}
\newcommand{\cdf}[2]{{\cDef}_{#2}^{#1}}

\newcommand{\fr}[1]{\mathfrak{{#1}}}

\newcommand{\hm}[4]{{\Hom}_{#2}^{#1}({#3},{#4})}
\newcommand{\uhm}[4]{{\uHom}_{#2}^{#1}({#3},{#4})}

\newcommand{\mc}[1]{\mathcal{#1}}

\newcommand{\mr}[1]{\mathrm{{#1}}}
\newcommand{\ms}[1]{\mathscr{{#1}}}
\newcommand{\nd}[3]{{\End} _{#2}^{#1}({#3})}

\newcommand{\Q}{\mathcal{O}}

\newcommand{\tor}[4]{{\Tor}_{#2}^{#1}({#3},{#4})}
\newcommand{\ul}[1]{\underline{{#1}}}
\newcommand{\xla}[1]{\xleftarrow{{#1}}}
\newcommand{\xra}[1]{\xrightarrow{{#1}}}
\newcommand{\xt}[4]{{\Ext} _{#2}^{#1}({#3},{#4})}

\newcommand{\syz}[2]{{\Syz}_{#2}^{#1}}

\begin{document}
\title[Stably reflexive modules]
{Stably reflexive modules \\and a Lemma of Knudsen}
\author{Runar Ile}
\address{Department of 
Mathematics \\University of Bergen \\
5008 Bergen \\ NORWAY}
\email{runar.ile@math.uib.no} 
\thanks{\textit{Acknowledgement}. Part of this work was done during the author's pleasant stay at the Royal Institute of Technology in Stockholm} 
\begin{abstract} 
In his fundamental work on the stack \(\bar{\ms{M}}_{g,n}\) of stable \(n\)-pointed genus \(g\) curves, Finn F.\ Knudsen introduced the concept of a stably reflexive module in order to prove a key technical lemma. We propose an alternative definition and generalise the results in his appendix to \cite{knu:83a}. Then we give a `coordinate free' generalisation of his lemma, generalise a construction used in Knudsen's proof concerning versal families of pointed algebras, and show that Knudsen's stabilisation construction works for plane curve singularities. In addition we prove approximation theorems generalising Cohen-Macaulay approximation with stably reflexive modules in flat families. The generalisation is not covered (even in the closed fibres) by the Auslander-Buchweitz axioms.
\end{abstract}
\maketitle
\section{Introduction}
In order to establish the stabilisation map from the universal curve \(\bar{\ms{C}}_{g,n}\) to the stack \(\bar{\ms{M}}_{g,n+1}\) Knudsen developed in \cite{knu:83a} a theory of what he called stably reflexive modules with respect to a flat ring homomorphism. In this article we take a closer look at this theory and some of its applications both in approximation theory of modules and in deformation theory.

We reformulate Knudsen's theory by first defining an absolute notion of a stably reflexive module over a ring. In the noetherian case this is the same as a module of Gorenstein dimension \(0\). Then we define a stably reflexive module with respect to a flat ring homomorphism as a flat family of stably reflexive modules. By cohomology-and-base-change theory this definition is equivalent to Knudsen's.

We introduce a descending series of additive categories of modules, called \emph{\(n\)-stably reflexive modules}. For \(n=1\) these are the reflexive modules. The stably reflexive modules are \(n\)-stably reflexive for all \(n\). We also define an absolute and a relative notion of \emph{\(n\)-stably reflexive complexes}. An \(n\)-stably reflexive complex \((E^{\bdot},d^{\bdot})\) gives an \(n\)-stably reflexive module \(\coker d^{-1}\). In the noetherian case the stably reflexive complexes are Tate resolutions relative to a base ring. Finally, there is a third concept of a flat family of  \emph{\(n\)-orthogonal  modules} defined by a one-sided cohomology condition. An \(2n\)-orthogonal  module determines and is determined by an \(n\)-stably reflexive module  through the \((n{+}1)\)-syzygy. Together Propositions \ref{prop.main}, \ref{prop.refcplx} and \ref{prop.ortho} extend Theorem 2 in \cite[Appendix]{knu:83a}; see Remark \ref{rem.syzref}.

Axiomatic Cohen-Macaulay approximation was introduced by M. Auslander and R.-O.\ Buchweitz in \cite{aus/buc:89}. This theory has recently been defined in terms of fibred categories as to give approximation results for various classes of flat families of modules; see \cite{ile:12a}.  
In Theorem \ref{thm.approx3} we extend the approximation results of Auslander and Buchweitz in the classical case of modules of finite Gorenstein dimension to the relative setting: Let \(h\co S\ra R\) be a faithfully flat finite type algebra of noetherian rings and let \(N\) be an \(S\)-flat finite \(R\)-module. If \(N\) is stably reflexive w.r.t.\ \(h\), respectively has finite projective dimension as \(R\)-module, then \(N\) belongs to the module subcategory \(\cat{X}_{h}\), respectively \(\Pf_{h}\). If some syzygy \(\syz{R}{i}N\) is stably reflexive w.r.t.\ \(h\) then \(N\) belongs to the category \(\cat{Y}_{h}\). For any \(N\) in \(\cat{Y}_{h}\) there are short exact sequences of \(R\)-modules
\begin{equation}\label{eq.A}
0\ra L\lra M\lra N\ra 0\qquad\text{and}\qquad 0\ra N\lra L'\lra M'\ra 0
\end{equation}
with \(M\) and \(M'\) in \(\cat{X}_{h}\) and \(L\) and \(L'\) in \(\Pf_{h}\). 
The map \(M\ra N\) is a right \(\cat{X}_{h}\)-approximation and the map \(N\ra L'\) is a left \(\Pf_{h}\)-approximation. In particular \(\cat{X}_{h}\) is contravariantly finite in \(\cat{Y}_{h}\) and \(\Pf_{h}\) is covariantly finite in \(\cat{Y}_{h}\). In the local case there are minimal (and hence unique) approximations \eqref{eq.A}. 
The association \(N\mapsto X\) for \(X\) equal to \(M\) and \(L'\) induce functors on the stable quotient categories. They are adjoint to the natural inclusions. Moreover, the approximations and the functors are well behaved w.r.t.\ base change \(S\ra S'\); see Theorem \ref{thm.approx3}. In Theorems \ref{thm.approx1} and \ref{thm.approx2} we give analogous results for larger classes of modules \(N\) by using \(n\)-stably reflexive modules in the approximations. The category \(\Pf_{h}\) is replaced by the category \(\Pf(s)_{h}\) of modules of projective dimension less than or equal to \(s\) and \(\cat{Y}_{h}\) is replaced by categories depending on \(n\) and \(s\). Note that these categories do not satisfy the axioms of Auslander and Buchweitz. We refer to the introduction of \cite{ile:12a} for a discussion and further references on Cohen-Macaulay approximation.

The stabilisation map takes a stable \(n\)-pointed curve \(C\) with an extra section \(\Delta\) to a stable \((n+1)\)-pointed curve \(C^{s}\) \cite[2.4]{knu:83a}. Knudsen's lemma is applied in the critical case where \(\Delta\) hits a node in order to obtain flatness, functoriality of the construction of \(C^{s}\), and the existence of a functorial lifting of \(\Delta\). While Knudsen both in his original and his recent account applies several explicit calculations for a quadratic form to prove this; see \cite{knu:83a,knu:12}, a main benefit of our approach to the stably reflexive modules is how readily his results are generalised. Knudsen's lemma \cite[2.2]{knu:83a}, which deals with deformations of a \(1\)-dimensional ordinary double point, is generalised in Theorem \ref{prop.Finn} to any pointed deformation \(S\ra R\ra S\) of a \(1\)-dimensional Gorenstein ring \(A=R/\fr{m}_{S}R\) defined over an arbitrary field. 
Theorem \ref{prop.Finn} says that the ideal \(I\) defining the point; \(R/I\cong S\), as module is stably reflexive w.r.t.\ \(S\ra R\) and \(\hm{}{R}{I}{R}/R\) is isomorphic to \(S\). By the established results on stably reflexive modules these properties are preserved by base change. In particular these two results imply functoriality of the stabilisation construction as explained in Section \ref{sec.stab}. Theorem \ref{prop.Finn} also says that \(\hm{}{R}{I}{R}\) is isomorphic to the fractional ideal \(\{f\in K(R)\,\vert\, f{\cdot} I\sbeq R\}\), and determines the image of \(\fr{m}_{A}\ot\hm{}{A}{\fr{m}_{A}}{A}\ra A\). Knudsen applies the former in his proof of the key lemma; see \cite[2.2]{knu:83a} and \cite{knu:12}, and the latter implies that stabilisation inserts a \(\BB{P}^{1}\) if the point is singular. 

In order to establish flatness of \(C^{s}\), Knudsen applies deformation theory. For a \(1\)-dimensional ordinary double point he finds an explicit formally versal formal family of pointed local rings for which he analyses his construction of \(C^{s}\); see his recent account \cite{knu:12}. In Theorem \ref{thm.main} we show quite generally that the `square' of a versal family for deformations of a ring together with the `diagonal' gives a versal family for deformations of the pointed ring. In Corollary \ref{cor.main2} we calculate the versal family for a pointed isolated complete intersection singularity. In Proposition \ref{prop.mf2} we extend Knudsen's stabilisation construction to any flat family of plane curve singularities and obtain the relevant features. In the last section we explain the main steps in the proof of the stabilisation map \cite[2.4]{knu:83a} and how these results apply. 
\section*{Acknowledgement}
I thank Finn F.\ Knudsen who suggested to me, many years ago, to study his Lemma 2.2 in \cite{knu:83a}. In the summer 2009 he urged me again and this time I wrote up roughly the first half of this article. I believe none of us looked much more at it until Finn went to Ann Arbor in January 2011 where he took on his own independent, explicit approach, resulting in \cite{knu:12}. Encouraged, I completed my generalisation of the lemma. Finn kept sending me versions of his manuscript and the explicit and elegant construction of the hull inspired me to find the general Theorem \ref{thm.main} and Corollary \ref{cor.main2}.

I would also like to thank the anonymous referee for a detailed and useful report which contributed to a more readable article.
\section{Preliminaries}
All rings are unital and commutative. If \(A\) is a ring let \(\cat{Mod}_{A}\) denote the category of \(A\)-modules and \(\cat{mod}_{A}\) the category of finite \(A\)-modules.
\subsection{Coherent modules}
We follow \cite[Chap.\ 1, \S2, Exc.\ 11-12]{bou:98} and \cite[Chap.\ 2]{gla:89}.
Let \(A\) be a ring. An \(A\)-module \(N\) is \emph{coherent} if \(N\) is finite and if all finite submodules of \(N\) are finitely presented. A ring is coherent if it is coherent as a module over itself. 
All finitely presented modules over a ring \(A\) are coherent if and only if \(A\) is coherent. A polynomial ring over a noetherian ring is coherent. A coherent ring modulo a finitely generated ideal is coherent.
Let \(\cat{coh}_{A}\) denote the full subcategory of \(\cat{Mod}_{A}\) of coherent modules. It is an abelian category closed under tensor products and internal \(\Hom\). 
Assume \(A\) is coherent. Then any coherent module has a resolution by finite free modules. Hence if \(M\) and \(N\) are coherent then \(\tor{A}{i}{M}{N}\) and \(\xt{i}{A}{M}{N}\) are coherent \(A\)-modules for all \(i\). Moreover \(\mc{S}^{-1}M\) is a coherent \(\mc{S}^{-1}A\)-module for any multiplicatively closed subset \(\mc{S}\) of \(A\).
\subsection{Base change}
The main tool for reducing properties to the fibres in a flat family will be the base change theorem.
We follow the quite elementary and general approach of A.\ Ogus and G.\ Bergman \cite{ogu/ber:72}.
\begin{defn}
Let \(h\co S\ra R\) be a ring homomorphism and \(I\) an \(S\)-module.
Let \(F\) be an \(S\)-linear functor of some additive subcategory of \(\cat{Mod}_{S}\) containing \(S\) to \(\cat{Mod}_{R}\). Then the \emph{exchange map \(e_{I}\) for \(F\)} is defined as the \(R\)-linear map \(e_{I}\co F(S)\ot_{S}I\ra F(I)\) given by \(\xi\ot u\mapsto F(u)(\xi)\) where we consider \(u\) as the multiplication map \(u\co S\ra I\). Let \(\mSpec R\) denote the set of closed points in \(\Spec R\).
\end{defn}
\begin{prop}[{\cite[5.1-2]{ogu/ber:72}}]\label{prop.nakayama}
Let \(h\co S\ra R\) be a ring homomorphism with \(S\) noetherian\textup{.} Suppose \(\{F^{q}\co \cat{mod}_{S}\ra \cat{mod}_{R}\}_{q\in \BB{Z}}\) is an \(h\)-linear cohomological \(\delta\)-functor \textup{(}in particular \(F^{q}\) is the zero functor for \(q<0\)\textup{).} 
\begin{enumerate}
\item[(i)] If the exchange map \(e_{S/\fr{n}}^{q}\co F^{q}(S)\ot_{S}S/\fr{n}\ra F^{q}(S/\fr{n})\) is surjective for all \(\fr{n}\) in \(Z=\im\{\mSpec R\ra\Spec S\}\)\textup{,} then \(e_{I}^{q}\co F^{q}(S)\ot_{S}I\ra F^{q}(I)\) is an isomorphism for all \(I\) in \(\cat{mod}_{S}\)\textup{.} 
\item[(ii)] If \(e_{S/\fr{n}}^{q}\) is surjective for all \(\fr{n}\) in \(Z\)\textup{,} then \(e_{I}^{q-1}\) is an isomorphism for all \(I\) in \(\cat{mod}_{S}\) if and only if \(F^{q}(S)\) is \(S\)-flat\textup{.}
\end{enumerate}
\end{prop}
\begin{rem}
If the \(F^{q}\) in addition extend to functors of all \(S\)-modules \(F^{q}\co \cat{Mod}_{S}\ra\cat{Mod}_{R}\) which commute with filtered direct limits, then the conclusions are valid for all \(I\) in \(\cat{Mod}_{S}\).
If \(F^{q}(S/\fr{n})=0\) for all \(\fr{n}\in Z\) then \(F(S)\ot_{R}R/\fr{m}=0\) for all \(\fr{m}\in\mSpec R\) by (i) hence \(F(S)=0\) by Nakayama's lemma and so \(F^{q}=0\) by (i) again.
For \(q=0\) statement (ii) reads: `If \(e_{S/\fr{n}}^{0}\) is surjective for all \(\fr{n}\) in \(Z\)\textup{,} then \(F^{0}(S)\) is \(S\)-flat\textup{.}'
\end{rem} 
\begin{ex}\label{ex.nakcplx}
Suppose \(R\) is coherent.
Let \(K\co K^{0}\ra K^{1}\ra\dots\) be a complex of \(S\)-flat and coherent \(R\)-modules. Define \(F^{q}\co \cat{mod}_{S}\ra\cat{coh}_{R}\) by \(F^{q}(I)=\cH^{q}(K\ot_{S}I)\). Then \(\{F^{q}\}\) is an \(h\)-linear cohomological \(\delta\)-functor which extends to all \(S\)-modules and commutes with direct limits.
\end{ex}
Let \(\fr{m}_{A}\) denote the maximal ideal in a local ring \(A\).
\begin{cor}\label{cor.nak}
Assume that \(S\ra R\) is a flat and local homomorphism of noetherian rings and let \((G)=(G_{1},\dots,G_{n})\) be a sequence in the maximal ideal \(\fr{m}_{R}\)\textup{.} 

Then \((G)\) is a regular sequence and \(R/(G)\) is \(S\)-flat if and only if the image of \((G)\) in \(R/\fr{m}_{S}R\) is a regular sequence\textup{.}
\end{cor}
\begin{proof}
Let \(K\) be the Koszul complex of \((G)\) on \(R\) \textup{(}cohomologically graded and shifted\textup{).}
For the `only if' direction note that \(n\) is the maximal non-trivial degree of \(K\). Then \(e^{n}_{I}\) is an isomorphism for all \(I\) by Proposition \ref{prop.nakayama} (ii) with \(q=n+1\). Since \(F^{n}(S)=\cH^{n}(K)=R/(G)\) by assumption is \(S\)-flat, \(e^{n-1}_{S/\fr{m}_{S}}\) is an isomorphism by Proposition \ref{prop.nakayama} (ii) with \(q=n\). By assumption \(F^{n-1}(S)=0\) hence \(F^{n-1}(S/\fr{m}_{S})=0\). Since \(F^{n-1}(S)\) trivially is \(S\)-flat we can proceed by downward induction using Proposition \ref{prop.nakayama} (ii) to show that \(F^{i}(S/\fr{m}_{S})=0\) for all \(i<n\). Since \(R/\fr{m}_{S}R\) is a local ring, acyclic Koszul complex implies regular sequence.

Conversely, if the image of \((G)\) in \(R/\fr{m}_{S}R\) is a regular sequence, then \(F^{i}(S/\fr{m}_{S})=0\) for \(i\neq n\) and it follows that \(F^{i}(S)=0\) for \(i\neq n\) by Proposition \ref{prop.nakayama} (i). Since \(R\) is a local ring \((G)\) is a regular sequence. By Proposition \ref{prop.nakayama} (ii), \(F^{n}(S)=R/(G)\) is \(S\)-flat.
\end{proof}
\begin{ex}\label{ex.nakayama}
Suppose \(R\) is coherent. Let \(M\) and \(N\) be coherent \(R\)-modules with \(N\) \(S\)-flat. Then the functors \(F^{q}\co \cat{mod}_{S}\ra\cat{coh}_{R}\) defined by \(F^{q}(I)=\xt{q}{R}{M}{N\ot_{S}I}\) give an \(h\)-linear cohomological \(\delta\)-functor which extends to all \(S\)-modules and commutes with direct limits.
\end{ex}
Let \(S\ra R\) and \(S\ra S'\) be ring homomorphisms, \(M\) an \(R\)-module, \(R'=R\ot_{S}S'\) and \(N'\) an \(R'\)-module. Then there is a change of rings spectral sequence 
\begin{equation}\label{eq.ss}
\cE_{2}^{p,q}=\xt{q}{R'}{\tor{S}{p}{M}{S'}}{N'}\Ra\xt{p+q}{R}{M}{N'}\,.
\end{equation} 
In addition to the isomorphism \(\hm{}{R'}{M\ot_{S}S'}{N'}\cong\hm{}{R}{M}{N'}\) there are edge maps \(\xt{q}{R'}{M\ot_{S}S'}{N'}\ra\xt{q}{R}{M}{N'}\) for \(q>0\) which are isomorphisms too if \(M\) (or \(S'\)) is \(S\)-flat. If \(I'\) is an \(S'\)-module we can compose the exchange map \(e^{q}_{I'}\) (regarding \(I'\) as \(S\)-module) with the inverse of this edge map for \(N'=N\ot_{S}I'\) and obtain the \emph{base change map} \(c^{q}_{I'}\) of \(R'\)-modules 
\begin{equation}\label{eq.basechange}
c^{q}_{I'}\co \xt{q}{R}{M}{N}\ot_{S}I'\ra\xt{q}{R'}{M\ot_{S}S'}{N\ot_{S}I'}\,.
\end{equation}
We will use the following geometric notation. Suppose \(S\ra R\) is a ring homomorphism, \(M\) is an \(R\)-module and \(s\) is a point in \(\Spec S\) with residue field \(k(s)\). Then \(M_{s}\) denotes the fibre \(M\ot_{S}k(s)\) of \(M\) at \(s\) with its natural \(R_{s}=R\ot_{S}k(s)\)-module structure. Now Proposition \ref{prop.nakayama} implies the following:
\begin{cor}\label{cor.xtdef}
Let \(S\ra R\) and \(S\ra S'\) be ring homomorphisms with \(S\) noetherian and \(R\) coherent\textup{.} Suppose \(M\) and \(N\) are coherent \(R\)-modules\textup{,} \(Z=\im\{\mSpec R\ra\Spec S\}\) and \(q\) is an integer\textup{.} Assume that \(M\) and \(N\) are \(S\)-flat\textup{.}
\begin{enumerate}
\item[(i)] If\, \(\xt{q+1}{R_{s}}{M_{s}}{N_{s}}=0\) for all \(s\) in \(Z\)\textup{,} then \(c^{q}_{I'}\) in \eqref{eq.basechange} is an isomorphism for all \(S'\)-modules \(I'\)\textup{.}
\item[(ii)] If in addition \(\xt{q-1}{R_{s}}{M_{s}}{N_{s}}=0\) for all \(s\in Z\)\textup{,} then \(\xt{q}{R}{M}{N}\) is \(S\)-flat\textup{.}
\end{enumerate}
\end{cor}
\subsection{Reflexive modules}
Let \(A\) be a ring and \(M\) an \(A\)-module. The dual of \(M\) is the \(A\)-module \(\hm{}{A}{M}{A}\) which we denote by \(M^{*}\). There is a natural map \(\sigma_{M}\co M\ra M^{**}\) given by \(\sigma_{M}(m)=\ev_{m}\), the evaluation map. The module is called \emph{torsionless} if \(\sigma_{M}\) is injective and \emph{reflexive} if \(\sigma_{M}\) is an isomorphism; cf.\ \cite[1.4]{bru/her:98} for finite modules and noetherian rings. Suppose \(G\xra{d} F\ra M\ra 0\) is a projective presentation of \(M\).
Define \(D(M)\), the \emph{transpose of \(M\)}, to be the cokernel of the dual \(d^{*}\co F^{*}\ra G^{*}\). Put \(\Syz M=\im d\) and \(\Syz D(M)=\im d^{*}\).
\begin{lem}[cf.\ {\cite[1.4.21]{bru/her:98}}]\label{lem.D}
There are canonical isomorphisms 
\begin{equation*}
\ker\sigma_{M}\cong\xt{1}{A}{D(M)}{A}\quad \text{and}\quad \coker\sigma_{M}\cong\xt{2}{A}{D(M)}{A}\textup{.}
\end{equation*}
\end{lem}
\begin{proof}
The projective presentation induces a commutative diagram with exact rows
\begin{equation}\label{eq.D}
\xymatrix@-0pt@C-0pt@R-8pt@H-0pt{
0\ar[r] & (\syz{}{}(DM))^{*} \ar[r] & F\ar[r]^(0.45){\psi} & M^{**}\!\!\!\!\!\! \\
0\ar[r] & \syz{}{}M \ar[r]\ar[u]^(0.45){\rho} & F\ar[r]\ar@{=}[u] & M\ar[r]\ar[u]_{\sigma_{M}} & 0
}
\end{equation}
where \(\rho\) is the inclusion \(\im d\sbeq(\im d^{*})^{*}=(\syz{}{}(DM))^{*}\). Then \(\coker \sigma_{M}\cong\coker\psi\linebreak[2]\cong\xt{1}{A}{\syz{}{}(DM)}{A}\cong \xt{2}{A}{D(M)}{A}\) and \(\ker \sigma_{M}\cong \coker\rho\cong \xt{1}{A}{D(M)}{A}\).
\end{proof}
Assume \(A\) is a coherent ring and recall the \emph{stable category} \(\ucoh_{A}\) defined by A.\ Heller in \cite{hel:60}. It has the same objects as \(\cat{coh}_{A}\) and morphisms \(\uhm{}{A}{M}{N}\) given as the quotient \(\hm{}{A}{M}{N}/{\sim}\) of the \(A\)-linear homomorphisms by the following equivalence relation: Maps \(f\) and \(g\) in \(\hm{}{A}{M}{N}\) are stably equivalent if \(f{-}g\) factors through a coherent projective \(A\)-module. Subcategories of the stable category will appear in connection with the approximation results in Section \ref{sec.approx}.

In the stable category the syzygy can be made into a functor as follows. Fix a projective presentation \(G\xra{d} F\ra M\ra 0\) for each coherent \(A\)-module \(M\), i.e.\ a presentation where \(G\) and \(F\) are coherent and projective \(A\)-modules. The (first) \emph{syzygy module} \(\syz{A}{}M\) of \(M\) is defined to be \(\im d\). Inductively define \(\syz{A}{n}M=\syz{A}{}(\syz{A}{n-1}M)\) and put \(\syz{A}{0}M=M\). 
As the syzygy only depends on the choices made up to stable equivalence \(\syz{A}{}\) induces an endo-functor on \(\ucoh_{A}\). One has \(\uhm{}{A}{\syz{A}{n}M}{N}\cong\xt{n}{A}{M}{N}\) for all \(n>0\). 
Similarly \(D\) defines a duality on \(\ucoh_{A}\).  In the following we will allow \(D(M)\) and \(\syz{A}{n}(M)\) to denote any representative for the corresponding stable isomorphism class. 

Given a short exact sequence \(\xi\co 0\ra M_{1}\ra M_{2}\ra M_{3}\ra 0\) in \(\cat{coh}_{A}\) there is an exact sequence of projective presentations of \(\xi\). After dualisation the snake lemma gives an exact sequence in \(\cat{coh}_{A}\):
\begin{equation}\label{eq.exD}
0\ra M_{3}^{*}\ra M_{2}^{*}\ra M_{1}^{*}\ra D(M_{3})\ra D(M_{2})\ra D(M_{1})\ra 0
\end{equation}
\begin{lem}\label{lem.Dref}
Suppose \(A\) is a coherent ring and \(\xi\co  0\ra M_{1}\ra M_{2}\ra M_{3}\ra 0\) is a short exact sequence of coherent \(A\)-modules\textup{.} Assume that the natural map \(\xt{1}{A}{M_{3}}{A}\ra \xt{1}{A}{M_{2}}{A}\) is injective\textup{.} Then \(M_{1}^{*}\ra D(M_{3})\) in \eqref{eq.exD} is the zero map and the short exact sequence \(D(\xi)\) induces the exact sequence
\begin{align*}\label{eq.Dseq}
0\lra & \,\xt{1}{A}{D(M_{1})}{A}\lra\xt{1}{A}{D(M_{2})}{A}\lra\xt{1}{A}{D(M_{3})}{A}\lra \\
& \,\xt{2}{A}{D(M_{1})}{A}\lra\xt{2}{A}{D(M_{2})}{A}\lra\xt{2}{A}{D(M_{3})}{A}\,.
\end{align*}
The last map is surjective if \(\xt{1}{A}{M_{1}^{*}}{A}\ra\xt{1}{A}{M_{2}^{*}}{A}\) is injective\textup{.}
\end{lem}
\begin{proof}
The connecting map \(M_{1}^{*}\ra D(M_{3})\) factors through the connecting map \(M_{1}^{*}\ra\xt{1}{A}{M_{3}}{A}\). Since \(\xt{1}{A}{M_{3}}{A}\ra D(M_{3})\) is injective \(D(\xi)\) is a short exact sequence if and only if the connecting \(M_{1}^{*}\ra\xt{1}{A}{M_{3}}{A}\) is the \(0\)-map, which is true by assumption. Moreover, the dual of \(\xi\) gives a short exact sequence \(\xi^{*}\). Dualising once more gives a left exact sequence \(\xi^{**}\) and a natural map of complexes \(\sigma_{\xi}\co \xi\ra\xi^{**}\). Then the \(6\) terms from the long exact sequence of \(\hm{}{A}{-}{A}\) applied to \(D(\xi)\) is naturally identified with the exact sequence obtained from the snake lemma of \(\sigma_{\xi}\).
\end{proof}
\begin{ex}\label{ex.G}
Let \(A\) be a Cohen-Macaulay ring (in particular noetherian) with a canonical module \(\omega_{A}\); cf.\ \cite[3.3.16]{bru/her:98}. For an \(A\)-module \(M\) let \(M^{v}\) denote the \(A\)-module \(\hm{}{A}{M}{\omega_{A}}\). By local duality theory the evaluation map \(M\ra M^{vv}\) is an isomorphism if \(M\) is a (finite) maximal Cohen-Macaulay (MCM) module. If \(A\) is Gorenstein then \(A\) is a canonical module for \(A\) and hence MCM \(A\)-modules are reflexive. 
\end{ex}
\section{$n$-stably reflexive modules}
We define an \(n\)-stably reflexive module with respect to a flat ring homomorphism. Knudsen's definition of a stably reflexive module in Theorem 2 of \cite[Appendix]{knu:83a} implies ours. Proposition \ref{prop.main} gives the converse.
\begin{defn}\label{def.nstab}
Let \(A\) be a ring and \(n\) a positive integer. An \(A\)-module \(M\) is \emph{\(n\)-stably reflexive} if \(M\) is reflexive and \(\xt{i}{A}{M}{A}=0=\xt{i}{A}{M^{*}}{A}\) for all \(0<i<n\). If \(M\) is \(n\)-stably reflexive for all \(n\) it is called a \emph{stably reflexive} module.
\end{defn}
A \(1\)-stably reflexive module is the same as a reflexive module. If \(M\) is reflexive so is \(M^{*}\).
Hence \(M^{*}\) is \(n\)-stably reflexive if \(M\) is. 
\begin{rem}\label{rem.Gdim}
Auslander and M.\ Bridger introduced \emph{Gorenstein dimension} in \cite{aus/bri:69}. For noetherian rings and finite modules what we here call a stably reflexive module is the same as a module of Gorenstein dimension \(0\); see \cite[3.8]{aus/bri:69}.
\end{rem}
\begin{ex}\label{ex.Gor}
As noted in Example \ref{ex.G} if \(A\) is Gorenstein (in particular noetherian) then a finite MCM \(A\)-module \(M\) is reflexive. By local duality theory \(\xt{i}{A}{M}{A}=0\) for all \(i>0\). Since \(M^{*}\) also is MCM, \(M\) is stably reflexive as an \(A\)-module. The converse is also true: Localising at a prime ideal \(\fr{p}\) in \( A\), Grothendieck's local duality theorem with \(\omega_{A_{\fr{p}}}= A_{\fr{p}}\) and \(\fr{m}=\fr{p}A_{\fr{p}}\) gives 
\begin{equation}\label{eq.ldt}
\cH^{i}_{\fr{m}}(M_{\fr{p}})\cong \hm{}{A_{\fr{p}}}{\xt{n{-}i}{A_{\fr{p}}}{M_{\fr{p}}}{A_{\fr{p}}}}{E(k)}\text{ for all } i 
\end{equation}
where \(E(k)\) is the injective hull of the residue field \(k\) and \(n=\dim A_{\fr{p}}\) (\cite[3.5.9]{bru/her:98}). Hence \(\cH^{i}_{\fr{m}}(M_{\fr{p}})=0\) for \(i<\dim A_{\fr{p}}\) for all prime ideals \(\fr{p}\) in \( A\) and \( M\) is a maximal Cohen-Macaulay \(A\)-module.
\end{ex}
We are now going to define a relative version of the above notion.
\begin{defn}\label{def.relnstab}
Let \(n\) be a positive integer. Suppose \(h\co S\ra R\) is a flat ring homomorphism and \(M\) an \(R\)-module. Then \(M\) is \emph{\(n\)-stably reflexive with respect to \(h\)} if \(M\) is \(S\)-flat and the fibre \(M_{s}=M\ot_{S}k(s)\) is \(n\)-stably reflexive as an \(R_{s}\)-module for all \(s\in\im\{\mSpec R\ra\Spec S\}\). The module \(M\) is \emph{stably reflexive} with respect to \(h\) if it is \(n\)-stably reflexive with respect to \(h\) for all \(n\). We will also say that \(M\) is \(n\)-stably (stably) reflexive \emph{over \(S\)}.
\end{defn}
\begin{prop}\label{prop.main}
Let \(h\co S\ra R\) be a flat ring homomorphism and \(M\) an \(S\)-flat coherent \(R\)-module\textup{.} Suppose \(S\) is noetherian\textup{,} \(R\) is coherent and \(n>1\)\textup{.} Assume the module \(M\) is \(n\)-stably reflexive with respect to \(h\)\textup{.}
\begin{enumerate}
\item[(i)] For any ring homomorphism \(S\ra S'\) the base change \(M'=M\ot_{S}S'\) is \(n\)-stably reflexive with respect to \(h'=h\ot_{S}S'\) and as \(R'=R\ot_{S}S'\)-module\textup{.}
\item[(ii)] The dual \(M^{*}=\hm{}{R}{M}{R}\) is \(n\)-stably reflexive with respect to \(h\)\textup{.}
\item[(iii)] For any ring homomorphism \(S\ra S'\) and \(S'\)-module \(I'\) the base change map \(M^{*}\ot_{S}I'\ra\hm{}{R'}{M'}{R'\ot_{S'}I'}\) is an isomorphism and 
\begin{equation*}
\xt{i}{R'}{M'}{R'\ot_{S'}I'}=0\quad \text{for all}\quad 0<i<n\textup{.}
\end{equation*}
\end{enumerate}
\end{prop}
\begin{proof}
To prove that \(M'\) is reflexive we first investigate the base change properties of \(D(M)\). Choose a projective presentation \(G_{1}\ra G_{0}\ra M\ra 0\) of \(M\). Let \(I\) denote an \(S\)-module. Define the cohomological \(\delta\)-functor \(\{F^{q}\}\) by \(F^{q}(I)=\cH^{q}((G_{0}^{*}\ra G_{1}^{*})\ot_{S}I)\) as in Example \ref{ex.nakcplx}. Note that \(F^{1}(S)=D(M)\). Since the \(G_{i}\) are coherent and projective, \(G_{i}^{*}\ot_{S}I\cong \hm{}{R}{G_{i}}{R\ot_{S}I}\) and it follows that \(F^{0}(I)\cong \hm{}{R}{M}{R\ot_{S}I}\). By assumption \(n>1\) and \(\xt{1}{R_{s}}{M_{s}}{R_{s}}=0\) for all \(s\in Z:=\im\{\mSpec R\ra\Spec S\}\). Since \(M\) is \(S\)-flat Corollary \ref{cor.xtdef} implies that the base change maps for \(F^{0}\) are isomorphisms, which is the first part of statement (iii). Since \(F^{2}=0\) the exchange maps for \(F^{1}\) are isomorphisms by Proposition \ref{prop.nakayama} (ii), in particular \(D(M)\ot_{S}S'\cong D(M')\). Moreover, \(D(M)\) is \(S\)-flat by Proposition \ref{prop.nakayama} (ii). Then the base change map \(c^{i}_{I'}\co \xt{i}{R}{D(M)}{R}\ot_{S}I'\ra \xt{i}{R'}{D(M')}{R'\ot_{S'}I'}\) is well defined. Since \(M_{s}\) is reflexive for \(s\in Z\), Lemma \ref{lem.D} implies that \(\xt{i}{R_{s}}{D(M_{s})}{R_{s}}=0\) for \(i=1,2\). By Proposition \ref{prop.nakayama} (i) this implies that \(\xt{i}{R'}{D(M')}{R'\ot_{S'}I'}=0\) for \(i=1,2\). In particular \(M'\) is reflexive as \(R'\)-module. Let \(s'\) be an element in \(Z':=\im\{\mSpec R'\ra\Spec S'\}\). The above conclusion for the composed ring homomorphism \(S\ra S'\ra k(s')\) says that \(M'_{s'}\) is reflexive as \(R'_{s'}\)-module. As \(M'\) is \(S'\)-flat it follows that \(M'\) is \(1\)-reflexive with respect to \(h'\).

Since \(\xt{i}{R_{s}}{M_{s}}{R_{s}}=0\) for all \(0<i<n\) and \(s\in Z\), Proposition \ref{prop.nakayama} (i) implies that \(\xt{i}{R}{M}{R}=0\) and that the base change map
\begin{equation}
c^{i}_{I'}(M)\co \xt{i}{R}{M}{R}\ot_{S}I'\lra\xt{i}{R'}{M'}{R'\ot_{S'} I'}
\end{equation}
is an isomorphism for all \(0<i<n\). In particular \(\xt{i}{R'}{M'}{R'\ot_{S'}I'}=0\) for \(0<i<n\) which gives the second part of (iii). The argument applied to the composition \(S\ra S'\ra k(s')\) for some \(s'\in Z'\) gives \(\xt{i}{{R'\!\!}_{s'}}{M_{s'}'}{R_{s'}'}=0\) for \(0<i<n\). 

As proved above \(M^{*}\) is \(S\)-flat and \(M^{*}\ot_{S}I'\linebreak[1]\cong \hm{}{R'}{M'}{R'\ot_{S'}I'}\). In particular \(M^{*}\ot_{S}S'\cong (M')^{*}\). The base change map
\begin{equation}\label{eq.bcS}
c^{i}_{I'}(M^{*})\co \xt{i}{R}{M^{*}}{R}\ot_{S}I'\lra\xt{i}{R'}{(M')^{*}}{R'\ot_{S'} I'}
\end{equation}
is therefore well-defined. Since \(\xt{i}{R_{s}}{(M_{s})^{*}}{R_{s}}=0\) for all \(0<i<n\) and \(s\in Z\) by assumption and \((M^{*})_{s}\cong (M_{s})^{*}\), Proposition \ref{prop.nakayama} (i) implies that \(\xt{i}{R}{M^{*}}{R}=0\) and that \(c^{i}_{I'}\) is an isomorphism for all \(0<i<n\). In particular \(\xt{i}{R'}{(M')^{*}}{R'}=0\) for all \(0<i<n\). Moreover, this conclusion for the composition \(S\ra S'\ra k(s')\) with \(s'\in Z'\) gives
\(\xt{i}{{R'\!\!}_{s'}}{({M'\!\!}_{s'})^{*}}{{R'\!}_{s'}}=0\) for \(0<i<n\). This concludes the proof of (i).

Now (ii) follows from the identification \((M^{*})_{s}\cong (M_{s})^{*}\) for all \(s\in Z\) proved above. Since \(M_{s}\) is reflexive, so is its dual \((M^{*})_{s}\). The vanishing of \(\xt{i}{R_{s}}{(M^{*})_{s}}{R_{s}}\) and of \(\xt{i}{R_{s}}{((M^{*})_{s})^{*}}{R_{s}}\) for \(0<i<n\) follows likewise from (i) and the assumptions, respectively.
\end{proof}
\begin{rem}\label{rem.main}
Knudsen gives three equivalent definitions of a stably reflexive module \(M\) with respect to a flat homomorphism of noetherian rings \(h:S\ra R\) in Theorem 2 of \cite[Appendix]{knu:83a}. We show that the first one is equivalent to our in Definition \ref{def.relnstab}.
\begin{defn*}[{\cite[Appendix]{knu:83a}}]
A finite \(R\)-module \(M\) is a stably reflexive module with respect to \(h\) if for all \(S\)-modules \(I\) 
\begin{enumerate}
\item[(a)] 
the exchange map \(\hm{}{R}{M}{R}\ot_{S}I\ra\hm{}{R}{M}{R\ot_{S}I}\) is an isomorphism,
\item[(b)] the composition
\begin{equation*}
M\ot_{S}I\xra{\sigma_{M}\ot\id_{I}}\hm{}{R}{M^{*}}{R}\ot_{S}I\ra\hm{}{R}{M^{*}}{R\ot_{S}I}
\end{equation*}
is an isomorphism, and
\item[(c)] \(\xt{i}{R}{M}{R\ot_{S}I}=0=\xt{i}{R}{M^{*}}{R\ot_{S}I}\) for all \(i>0\).
\end{enumerate}
\end{defn*}
Assume this. By (a) and Proposition \ref{prop.nakayama} (ii), \(M^{*}\) is \(S\)-flat. By (b), \(M\) is reflexive and the first map in (b) is an isomorphism. Hence the second map is an isomorphism. This implies by Proposition \ref{prop.nakayama} (ii) that \(M\) is \(S\)-flat. Pick \(s\in \im\{\mSpec R\ra\Spec S\}\) and put \(I=k(s)\).  Since \(M\) is \(S\)-flat the edge maps \(\xt{i}{R_{s}}{M_{s}}{R_{s}}\ra\xt{i}{R}{M}{R_{s}}\) in \eqref{eq.ss} are isomorphisms. By (c), \(\xt{i}{R_{s}}{M_{s}}{R_{s}}\linebreak[1]=0\) follows for all \(i>0\). By (a), \((M^{*})_{s}\) equals \((M_{s})^{*}\). Then the same argument gives \(\xt{i}{R_{s}}{(M_{s})^{*}}{R_{s}}=0\) for all \(i>0\). Moreover, (b) now gives \(M_{s}\cong\hm{}{R_{s}}{(M^{*})_{s}}{R_{s}}\cong (M_{s})^{**}\). Hence Knutsen's definition for \(M\) stably reflexive with respect to \(h\) implies our in Definition \ref{def.relnstab}. By Proposition \ref{prop.main} the reverse implication is true as well.
\end{rem}
\begin{lem}\label{lem.Gor}
Let \(h\co S\ra R\) be a flat homomorphism of noetherian rings with finite Krull dimension\textup{.} Let \(M\) be an \(S\)-flat finite \(R\)-module\textup{.} Suppose \(R_{s}\) is Gorenstein for all \(s\) in \(Z=\im\{\mSpec R\ra\Spec S\}\)\textup{.} Then the following statements are equivalent\textup{:}
\begin{enumerate}
\item[(i)] \(M\) is stably reflexive with respect to \(h\)\textup{.}
\item[(ii)] \(M_{s}\) is a maximal Cohen-Macaulay \(R_{s}\)-module for all \(s\in Z\)\textup{.}
\item[(iii)] \(\xt{i}{R_{s}}{M_{s}}{R_{s}}=0\) for all \(0<i\leq \dim R_{s}\) and \(s\in Z\)\textup{.}
\end{enumerate}
\end{lem}
\begin{proof}
Let \((R',M')\) denote the localisation of \((R_{s},M_{s})\) at a prime ideal \(\fr{p}\) in \(R_{s}\). Since \(R_{s}\) is noetherian and \(M_{s}\) is finite localising at \(\fr{p}\) commutes with \(\Ext\), i.e.\ \(\xt{i}{R_{s}}{M_{s}}{R_{s}}_{\fr{p}}\cong\xt{i}{R'}{M'}{R'}\). Since \(R'\) is a Gorenstein ring \(R'\) is a canonical \(R'\)-module. Hence vanishing of \(\xt{i}{R_{s}}{M_{s}}{R_{s}}\) for all \(i>0\) is equivalent to \(\depth M'=\dim R'\) for all \(\fr{p}\) which is equivalent to \(\xt{i}{R_{s}}{M_{s}}{R_{s}}=0\) for \(0<i\leq \dim R_{s}\); see \eqref{eq.ldt} and \cite[3.5.11]{bru/her:98}. This gives (i)\(\Ra\)(ii) and (ii)\(\lRa\)(iii). For (ii)\(\Ra\)(i) see Example \ref{ex.Gor}.
\end{proof}
The following lemma generalises Lemma and Corollary 3 in \cite[Appendix]{knu:83a}.
\begin{lem}\label{lem.nstabex2}
Let \(S\ra R\) be a flat ring homomorphism and \(\xi\co 0\ra{}^{1\!}M\ra {}^{2\!}M\ra {}^{3\!}M\ra 0\) a short exact sequence of coherent \(R\)-modules\textup{.} Suppose \(R\) is coherent and \(S\) is noetherian\textup{.}
\begin{enumerate}
\item[(i)] If \({}^{1\!}M\) and \({}^{3\!}M\) are \(n\)-stably reflexive over \(S\)\textup{,} so is \({}^{2\!}M\)\textup{.}
\item[(ii)] If \({}^{2\!}M\) and \({}^{3\!}M\) are \(n\)-stably reflexive over \(S\) and \(n>1\) then \({}^{1\!}M\) is \((n{-}1)\)-stably reflexive over \(S\)\textup{.}
\end{enumerate}
\end{lem}
\begin{proof}
In (i) and (ii) all three modules are \(S\)-flat. Let \(s\in\im\{\mSpec R\ra\Spec S\}\)\textup{.} As \(A=R_{s}\) is obtained by first localising and then dividing by a coherent ideal, it is a coherent ring and similarly for the modules. For (ii), since \(\xt{1}{A}{{}^{3\!}M_{s}}{A}=0\) Lemmas \ref{lem.D} and \ref{lem.Dref} give that \({}^{1\!}M_{s}\) is reflexive and the dual sequence \(\xi_{s}^{*}\) is short exact. Since \({}^{3\!}M_{s}\) is reflexive the double dual \(\xi_{s}^{**}\) is short exact too (consider \(\xi_{s}\ra\xi_{s}^{**}\)). In particular \(\xt{1}{A}{{}^{1\!}M_{s}^{*}}{A}=0\) as \(\xt{1}{A}{{}^{2\!}M_{s}^{*}}{A}=0\). The rest of the statement now follows from the long exact sequences. (i) is left to the reader.
\end{proof}
\section{$n$-stably reflexive complexes}
We define an \(n\)-stably reflexive complex with respect to a flat ring homomorphism \(S\ra R\). In the coherent case with \(n>1\) it is shown that this notion induces an \(n\)-stably reflexive module \(M\) and conversely, given \(M\), that such a complex exists. We also show that \(M\) is an \((n{+}1)\)-syzygy of an \(S\)-flat finite \(R\)-module with a one-sided cohomology condition. This too has a converse. 
\begin{defn}\label{def.refcplx}
Let \(n\) be a positive integer and \(A\) a ring.
\begin{enumerate}
\item[(i)] Let \(E\co  \dots\ra E^{-1}\xra{d^{-1}_{E}} E^{0}\xra{d^{0}_{E}} E^{1}\xra{d^{1}_{E}}\dots\) be a complex of finite projective \(A\)-modules. Let \(E^{\vee}\) denote the (non-standard) \emph{dual complex} where \((E^{\vee})^{i}=(E^{1{-}i})^{*}\) and the differential \(d^{i}_{E^{\vee}}\co (E^{\vee})^{i}\ra (E^{\vee})^{i+1}\) equals \((d^{-i}_{E})^{*}\co (E^{1-i})^{*}\ra (E^{-i})^{*}\) for all \(i\). If \(M\cong\coker d^{-1}_{E}\) then \(E\) is a \emph{hull} for the \(A\)-module \(M\). If \(\cH^{i}(E)=0=\cH^{i}(E^{\vee})\) for all \(i\leq n\) then \(E\) is an \emph{\(n\)-stably reflexive} \(A\)-complex. If \(E\) is \(n\)-stably reflexive for all \(n\) then \(E\) is \emph{stably reflexive}.

\item[(ii)] Let \(h\co S\ra R\) be a flat ring homomorphism. A complex \(E\) of finite projective \(R\)-modules is \emph{\(n\)-stably reflexive with respect to \(h\)} if \(E_{s}=E\ot_{S}k(s)\) is an \(n\)-stably reflexive \(R_{s}=R\ot_{S}k(s)\)-complex for all \(s\in\im\{\mSpec R\ra\Spec S\}\). The complex \(E\) is \emph{stably reflexive with respect to \(h\)} if it is \(n\)-stably reflexive with respect to \(h\) for all \(n\).
\end{enumerate}
\end{defn} 
\begin{rem}\label{rem.refcplx}
The non-standard dual of complexes is used to obtain symmetric statements. It is also used in Knutsen's orginal article \cite{knu:83a}.
\end{rem}
\begin{prop}\label{prop.refcplx}
Let \(h\co S\ra R\) be a flat ring homomorphism and \(M\) a coherent \(R\)-module\textup{.} Suppose \(R\) is coherent\textup{,} \(S\) is noetherian and \(n>1\)\textup{.} 
\begin{enumerate}
\item[(i)] There exists a complex \(E\) of finite projective \(R\)-modules which is a hull for \(M\) such that\textup{:} 
If \(M\) is \(n\)-stably reflexive with respect to \(h\) then  so is \(E\)\textup{.}

\item[(ii)] Let \(E\) be an \(n\)-stably reflexive complex with respect to \(h\) and a hull for \(M\)\textup{.} 

\begin{enumerate}
\item The \(R\)-module \(M\) is \(n\)-stably reflexive with respect to \(h\)\textup{.}

\item For any ring homomorphism \(S\ra S'\) the base change \(E'=E\ot_{S}S'\)  is \(n\)-stably reflexive with respect to \(h'=h\ot_{S}S'\) and as \(R'=R\ot_{S}S'\)-complex\textup{,} and a hull for \(M'=M\ot_{S}S'\)\textup{.}
\end{enumerate}
\end{enumerate}
\end{prop}
\begin{proof}
(iib) Let \(s\) be an element in \(\im\{\mSpec R\ra\Spec S\}\). Note that \(E^{\vee}\ot_{S}S'\cong (E')^{\vee}\) and in particular \((E^{\vee})_{s}\cong (E_{s})^{\vee}\) since \(E\) consists of finite projective modules. 
Consider an \(h\)-linear cohomological \(\delta\)-functor defined by \(E^{\vee}\) as in Example \ref{ex.nakcplx} (technically some downward truncation is needed). By assumption \(\cH^{i}(E_{s}^{\vee})=0\) for \(i\leq n\) and by Proposition \ref{prop.nakayama} (i), \(\cH^{i}(E^{\vee}\ot_{S}I)=0\) for all \(S\)-modules \(I\) and \(i\leq n\). If \(I=S'\) then \(\cH^{i}((E')^{\vee})=0\) and if  \(I=k(s')\) then \(\cH^{i}(E'_{s'}{\!\!}^{\vee})=0\) for any \(s'\in\im\{\mSpec R'\ra\Spec S'\}\) and \(i\leq n\). Replacing \(E^{\vee}\) by \(E\) in this argument gives \(\cH^{i}(E\ot_{S}I)=0\) and in particular \(\cH^{i}(E'_{s'})=0=\cH^{i}(E')\) for \(i\/\leq n\). Since \(\coker d^{-1}_{E'}\cong \coker d^{-1}_{E}\ot_{S}{\id}_{S'}= M'\), (b) is proved.

(iia) The map \(M\ot_{S}I\ra E^{1}\ot_{S}I\) induced by \(d_{E}^{0}\ot_{S}{\id}_{I}\) is injective since \(\cH^{0}(E\ot_{S}I)\linebreak[1]=0\). Consider the short exact sequence of coherent \(R\)-modules \(0\ra M\ra E^{1}\ra N\ra0\) where \(N\cong\coker d^{0}_{E}\). It remains short exact after applying \(-\ot_{S}I\) since \(\cH^{1}(E\ot_{S}I)=0\). As \(E^{1}\) is \(S\)-flat by assumption, \(N\) is \(S\)-flat and so \(M\) is \(S\)-flat too.

Dualising the exact sequence \(E^{-1}_{s}\ra E^{0}_{s}\ra M_{s}\ra 0\) gives the exact sequence \(0\ra M_{s}^{*}\ra (E_{s}^{\vee})^{1}\ra (E_{s}^{\vee})^{2}\), in particular \(M_{s}^{*}\cong\ker d_{E_{s}^{\vee}}^{1}\). Since \(\cH^{1}(E_{s}^{\vee})=0\), \(\ker d_{E_{s}^{\vee}}^{1}\) equals \(\im d_{E^{\vee}_{s}}^{0}\) and the dual of \(M_{s}\ra E^{1}_{s}\) is surjective. Since \(\cH^{0}(E_{s}^{\vee})=0\) the sequence \((E^{2}_{s})^{*}\ra (E^{1}_{s})^{*}\ra M_{s}^{*}\ra 0\) is exact. Dualising once more gives the exact sequence \(0\ra M_{s}^{**}\ra E^{1}_{s}\ra E^{2}_{s}\) and as \(M_{s}\cong \ker d_{E_{s}}^{1}\), \(M_{s}\) is reflexive.
Moreover \(\xt{i}{R_{s}}{M_{s}}{R_{s}}\cong\cH^{i{+}1}(E_{s}^{\vee})\) and \(\xt{i}{R_{s}}{M_{s}^{*}}{R_{s}}\cong\cH^{i{+}1}(E_{s})\) for \(i>0\) and so \(M\) is \(n\)-stably reflexive with respect to \(h\). 

(i) Since \(R\) and \(M\) are coherent, so is \(M^{*}\). Pick resolutions \(P\thr M\) and \(Q\thr M^{*}\) by finite projective modules. Splicing \(P\) with \(Q^{\vee}\) along \(P^{0}\thr M\ra M^{**}\hra (Q^{0})^{*}\) gives \(E\). By definition \(\cH^{i}(E_{s})\cong\tor{S}{-i}{M}{k(s)}\) for all \(i<0\) and since \(M\) is \(S\)-flat \(\cH^{i}(E_{s})=0\) for all \(i<0\). The resolution of \(M_{s}\) thus obtained implies that \(\cH^{i}(E_{s}^{\vee})\cong\xt{i-1}{R_{s}}{M_{s}}{R_{s}}\) for all \(i>1\). Hence by assumption \(\cH^{i}(E_{s}^{\vee})=0\) for \(1<i\leq n\) and \(\ker d_{E_{s}^{\vee}}^{1}\cong M_{s}^{*}\). By Proposition \ref{prop.main} (ii), \(M^{*}\) is \(n\)-stably reflexive over \(S\) too. In particular \(M^{*}\) is \(S\)-flat and as above \(\cH^{i}(E_{s}^{\vee})=0\) for \(i<0\) and \(\coker d_{E_{s}^{\vee}}^{-1}=d_{Q_{s}}^{-1}\cong M_{s}^{*}\) (note \((M^{*})_{s}\cong (M_{s})^{*}\) by Proposition \ref{prop.main} (iii)). Hence \(\cH^{i}(E_{s}^{\vee})\) vanishes for all \(i\leq n\). To obtain the symmetric statement \(\cH^{i}(E_{s})=0\) for all \(i\leq n\) note that \(M\cong M^{**}\) by Proposition \ref{prop.main} (i). We consider \(P\) as a resolution of \(M^{**}\). Splicing \(Q\) with \(P^{\vee}\) along \(Q^{0}\ra M^{*}\cong (M^{*})^{**}\ra (P^{0})^{*}\) gives \(E^{\vee}\). Now the argument for \(M\) and \(E\) above applies to \(M^{*}\) and \(E^{\vee}\).
\end{proof}
\begin{rem}\label{rem.refcplx2}
Knudsen's second characterisation (Theorem 2 (2) in \cite[Appendix]{knu:83a}) of a stably reflexive module with respect to a flat homomorphism of noetherian rings \(h\co S\ra R\) is in terms of a stably reflexive complex (Knudsen does not name this property).
\begin{defn*}[{\cite{knu:83a}, Appendix}]
A finite \(R\)-module \(M\) is a stably reflexive module with respect to \(h\) if there exists a complex of finite projective modules \(E\) which is a hull for \(M\) such that \(E\ot_{S}I\) and \(E^{\vee}\ot_{S}I\) are acyclic for all \(S\)-modules \(I\).
\end{defn*}
In particular the complex \(E\) is stably reflexive with respect to \(h\) as in Definition \ref{def.refcplx}. Conversely, a complex \(E\) stably reflexive with respect to \(h\) satisfies Knudsen's conditions by the proof of Proposition \ref{prop.refcplx} (iib). 
\end{rem}
The following notion will generalise Knutsen's in his third characterisation of a stably reflexive module.
\begin{defn}
Let \(A\) be a ring. An \(A\)-module \(N\) is \emph{left \(n\)-orthogonal to \(A\)} if \(\xt{i}{A}{N}{A}=0\) for all \(0<i\leq n\) and is \emph{left orthogonal to \(A\)} if it is left \(n\)-orthogonal to \(A\) for all \(n\). 

Let \(h\co S\ra R\) be a flat ring homomorphism and \(N\) an \(R\)-module. Then \(N\) is \emph{left \(n\)-orthogonal to \(h\)} if \(N\) is \(S\)-flat and \(N_{s}\) is left \(n\)-orthogonal to \(R_{s}\) for all \(s\in\im\{\mSpec R\ra \Spec S\}\). If \(N\) is left \(n\)-orthogonal to \(h\) for all \(n\) then \(N\) is \emph{left orthogonal to \(h\)}.
\end{defn}
\begin{prop}\label{prop.ortho}
Let \(h\co S\ra R\) be a flat ring homomorphism and \(M\) a coherent \(R\)-module\textup{.} Assume \(R\) is coherent\textup{,} \(S\) is noetherian and \(n\) is a positive integer\textup{.} Let \(E\) be an \(R\)-complex which is a hull for \(M\) as given in \textup{Proposition \ref{prop.refcplx} (i).}
\begin{enumerate}
\item[(i)] If \(N\) is a coherent \(R\)-module which is left \(2n\)-orthogonal to \(h\) then \(\syz{R}{n+1}N\) is \(n\)-stably reflexive with respect to \(h\)\textup{.}
\item[(ii)] Assume \(n>1\). The module \(M\) is \(n\)-stably reflexive with respect to \(h\) if and only if \(N:=\coker\{E^{n}\ra E^{n+1}\}\) is left \(2n\)-orthogonal to \(h\)\textup{.}
\item[(iii)] The module \(M\) is stably reflexive with respect to \(h\) if and only if \(N^{j}:=\coker\{E^{j-1}\ra E^{j}\}\) is left orthogonal to \(h\) for all \(j\in\BB{Z}\)\textup{.} In this case \(N^{j}\) is stably reflexive with respect to \(h\) for all \(j\)\textup{.}
\end{enumerate}
\end{prop}
\begin{proof}
(i) Let \(\dots\xra{d^{-2}} P^{-1}\xra{d^{-1}} P^{0}\ra N\ra 0\) be a resolution of \(N\) by finite projective modules. Put \({}^{1\!}M=\coker d_{P}^{-(n+2)}\) and let \(s\in Z:=\im\{\mSpec R\ra \Spec S\}\).  Since \(N\) is \(S\)-flat, \({}^{1\!}M=\syz{R}{n+1}N\) is \(S\)-flat too and, furthermore, \(P_{s}\) is a resolution of \(N_{s}\) by finite projective \(R_{s}\)-modules. For \(i>0\) one has \(\xt{i}{R_{s}}{{}^{1\!}M_{s}}{R_{s}}\cong\cH^{n+i+2}(P^{\vee}_{s})\cong \xt{n{+}i{+}1}{R_{s}}{N_{s}}{R_{s}}\) and the latter is \(0\) for \(0<i<n\) by assumption. Moreover, if \(L\) denotes \(\coker\{d_{P^{\vee}}^{2n+1}\co(P^{-2n})^{*}\ra (P^{-(2n{+}1)})^{*}\}\) the sequence (with \(d=d_{P_{s}^{\vee}}\))
\begin{equation}
0\la L_{s}\la (P_{s}^{-(2n{+}1)})^{*}\xla{d^{2n{+}1}}(P_{s}^{-2n})^{*}\la\dots \xla{d^{1}}(P_{s}^{0})^{*}\la N_{s}^{*}\la 0
\end{equation}
is exact. We use the complex \(K\co(P^{0})^{*}\ra\dots\ra (P^{-(2n+1)})^{*}\) to define a cohomological \(\delta\)-functor \(\{F^{q}\}\) as in Example \ref{ex.nakcplx}. Since \(F^{q}(k(s))=0\) for all \(0<q<2n+1\) and \(s\) in \(Z\), \(F^{q}(I)=0\) for all \(S\)-modules \(I\) and \(0<q<2n+1\) by Proposition \ref{prop.nakayama} (i). In particular the complex \(K\) gives a projective \((2n+1)\)-presentation of \(L=F^{2n+1}(S)\). By Proposition \ref{prop.nakayama} (ii), \(L\) is \(S\)-flat. Hence \({}^{1\!}M^{*}\cong \ker d_{P^{\vee}}^{n+2}\cong\syz{R}{n{+}1}L\) is \(S\)-flat. 
For \(0<i<n\) we get that \(\xt{i}{R_{s}}{{}^{1\!}M_{s}^{*}}{R_{s}}\cong\cH^{i-n+1}(K_{s}^{\vee})\cong\cH^{i-n}(P_{s})\). The latter is \(0\) if \(0<i<n\).
Since \({}^{1\!}M_{s}^{*}\cong\coker d_{P_{s}^{\vee}}^{n}\) reflexivity follows: 
\begin{equation}
{}^{1\!}M_{s}^{**}\cong \ker ((d_{P_{s}^{\vee}}^{n})^{*})\cong \ker d_{P_{s}}^{-n}\cong {}^{1\!}M_{s}
\end{equation}
(ii) One direction is (i). For the `only if' direction
note that 
\begin{equation}\label{eq.ortho}
\dots\lra E^{n}_{s}\lra E^{n+1}_{s}\lra N_{s}\ra 0
\end{equation}
is an \(R_{s}\)-projective resolution of \(N_{s}\) since \(E\) is \(n\)-stably reflexive over \(S\) by Proposition \ref{prop.refcplx} (i). Hence if \(i>0\) then \(\xt{i}{R_{s}}{N_{s}}{R_{s}}\cong\cH^{i-n}(E_{s}^{\vee})\). Then \(\xt{i}{R_{s}}{N_{s}}{R_{s}}=0\) for \(0<i\leq 2n\) since \(E\) is \(n\)-stably reflexive over \(S\). By Proposition \ref{prop.refcplx} (iib) \(E\) is an \(n\)-stably reflexive \(R\)-complex, so \(\dots\ra E^{n}\ra E^{n+1}\ra N\ra 0\) is an exact sequence. Applying \(-\ot_{S}k(s)\) gives the resolution in \eqref{eq.ortho} and thus, by Proposition \ref{prop.nakayama} (i), applying \(-\ot_{S}I\) for any \(S\)-module \(I\) gives a resolution of \(N\ot_{S}I\). Hence \(N\) is \(S\)-flat. 

(iii) By Proposition \ref{prop.refcplx} (i), \(M\) stably reflexive with respect to \(h\) implies that \(E\) is stably reflexive with respect to \(h\). In particular \(E_{s}^{\vee}\) is acyclic and hence \(N^{j}\) is left orthogonal to \(h\) for all \(j\). The reverse implication follows from (ii). As \(M=N^{0}\) the second part is clear by `translational symmetry'.
\end{proof}
\begin{ex}\label{ex.syzref}
Let \(h\co S\ra R\) be a flat homomorphism of noetherian rings. Let \(d\) denote the minimal depth at a maximal ideal of \(R_{s}=R\ot_{S}k(s)\) for all \(s\in Z=\im\{\mSpec R\ra \Spec S\}\). Suppose \(N\) is an \(S\)-flat finite \(R\)-module with \(\dim N_{s}=0\) for all \(s\in Z\). Then \(\syz{R}{n{+}1}N\) is \(n\)-stably reflexive with respect to \(h\) for all \(n>0\) with \(2n<d\) by Proposition \ref{prop.ortho} (i). If in addition \(n>1\) then \(N\) is left orthogonal to \(h\) by Proposition \ref{prop.ortho} (ii).
\end{ex}
\begin{rem}\label{rem.syzref}
Knudsen's third characterisation (Theorem 2 (3) in \cite[Appendix]{knu:83a}) of a stably reflexive module with respect to a flat homomorphism of noetherian rings \(h\co S\ra R\) is in terms of the syzygies of \(M\) (Knudsen does not name this property).
\begin{defn*}[{\cite{knu:83a}, Appendix}]
A finite \(R\)-module \(M\) is a stably reflexive module with respect to \(h\) if there exists an acyclic complex of finite projective modules \(E\) which is a hull for \(M\) such that \(N^{j}:=\coker\{E^{j-1}\ra E^{j}\}\) and \((N^{j})^{*}\) are left orthogonal to \(R\) and \(S\)-flat for all \(j\in\BB{Z}\)\textup{.}
\end{defn*}
Assume this. Let \(I\) be an \(S\)-module and \(j\) any integer. Applying \(-\ot_{S}I\) to the short exact sequence \(0\ra N^{j-1}\ra E^{j}\ra N^{j}\ra 0\) gives a short exact sequence since \(N^{j}\) is \(S\)-flat. This implies that \(E\ot_{S}I\) is acyclic. Dualising the short exact sequence gives a short exact sequence \(\xi\co0\ra (N^{j})^{*}\ra (E^{j})^{*}\ra (N^{j-1})^{*}\ra 0\) since \(\xt{1}{R}{N^{j}}{R}=0\). This implies that \(E^{\vee}\) is acyclic. Moreover, since \((N^{j-1})^{*}\) is \(S\)-flat, \(\xi\ot_{S}I\) is short exact. This implies that \(E^{\vee}\ot_{S}I\) is acyclic. Let \(s\) be any element in \(\im\{\mSpec R\ra\Spec S\}\). Since \(E_{s}\) and \(E_{s}^{\vee}\) are acyclic \(N^{j}_{s}\) is left orthogonal to \(R_{s}\). This shows that \(N^{j}\) is left orthogonal to \(h\) for all \(j\).

Conversely, given \(E\) such that \(N^{j}\) is left orthogonal to \(h\) for all \(j\) then the \(N^{j}\) are stably reflexive over \(S\) by Proposition \ref{prop.ortho} (iii). By Proposition \ref{prop.main} (i) the \(N^{j}\) are stably reflexive as \(R\)-modules. Hence \(E\) satisfies the properties in Knudsen's definition.

In view of Remarks \ref{rem.main} and \ref{rem.refcplx2}, we conclude that Knudsen's Theorem 2 in \cite[Appendix]{knu:83a} is generalised in Proposition \ref{prop.main}, \ref{prop.refcplx} and \ref{prop.ortho}.
\end{rem}
As an application we consider the (classical) case of a \(2\)-periodic complex associated to a matrix factorisation as introduced by D.\ Eisenbud in \cite{eis:80}.
\begin{defn}\label{def.mf}
Suppose \(T\) is a noetherian ring and \(f\) is a \(T\)-regular element. \emph{A matrix factorisation of \(f\)} is a pair \((\phi,\psi)\) of \(T\)-linear maps \(\phi\co G\ra F\) and \(\psi\co F\ra G\) between finite \(T\)-free modules such that \(\phi\psi=f{\cdot}\id_{F}\) and \(\psi\phi=f{\cdot}\id_{G}\).
\end{defn}
It follows that \(\phi\) and \(\psi\) are injective and that \(\Rk F=\Rk G\). It is also sufficient to check one of the equations.
\begin{cor}\label{cor.mf}
Let \(T\) be a noetherian ring and suppose \((\phi,\psi)\) is a matrix factorisation of a \(T\)-regular element \(f\)\textup{.} Put \(R=T/(f)\)\textup{.} Reduction by \(-\ot_{T}R\) of \((\phi,\psi)\) gives an acyclic \(2\)-periodic complex of free \(R\)-modules
\begin{equation*}
C(\phi,\psi)\co \quad\dots \xla{\bar{\psi}}\bar{F}\xla{\bar{\phi}}\bar{G}\xla{\bar{\psi}}\bar{F}\xla{\bar{\phi}}\bar{G}\xla{\bar{\psi}}\dots.
\end{equation*}
In addition\textup{,} suppose \(S\ra T\) is a flat ring homomorphism such that the image \(f_{s}\) of \(f\) is a \(T_{s}\)-regular element for all \(s\in Z=\im\{\mSpec T\ra\Spec S\}\)\textup{.} Then the induced ring homomorphism \(h\co S\ra T/(f)\) is flat\textup{.} Moreover\textup{;} \(C(\phi,\psi)\) is a stably reflexive complex with respect to \(h\) and a hull for the \(R\)-module \(\coker\phi\) which is stably reflexive with respect to \(h\)\textup{.}
\end{cor}
\begin{proof}
The first part is \cite[5.1]{eis:80}. Base change \((\phi_{s},\psi_{s})\) of \((\phi,\psi)\) to \(T_{s}\) is a matrix factorisation of \(f_{s}\). By Corollary \ref{cor.nak}, \(R\) is \(S\)-flat.
Since the dual of the complex \(C(\phi,\psi)\) is the complex of the dual maps \(C(\phi^{*},\psi^{*})\) we conclude from the first part of the statement and Proposition \ref{prop.refcplx}.
\end{proof}
\section{Approximation}\label{sec.approx}
We prove approximation theorems with \(n\)-stably reflexive modules resembling Cohen-Macaulay approximation in a flat family. The results generalise the classical setting of the Auslander-Buchweitz axioms in two directions. Firstly, the approximations are given in the relative setting with flat families of pairs (ring, module). This should make the results applicable to the local study of moduli, as we indeed show in the last two sections; cf.\ Remark \ref{rem.Finn1}. Secondly, we use approximation categories which do not satisfy the Auslander-Buchweitz axioms except in the stably reflexive case. Since the category of \(n\)-stably reflexive modules (strictly) contains the stably reflexive ones, the category of modules which can be approximated is enlarged too. Applications in deformation theory typically only need vanishing of cohomology in low degrees, e.g.\ \(n=3\) or \(n=4\) is sufficient. In general the parametres involved in the categories make the statements more precise.  

We will phrase our results in the language of fibred categories\footnote{We have chosen to work with rings instead of (affine) schemes. Our definition of a fibred category \(p\co\cat{F}\ra\cat{C}\) reflects this choice and is equivalent to the functor of opposite categories \(p^{\text{op}}\co\cat{F}^{\text{op}}\ra\cat{C}^{\text{op}}\) being a fibred category as defined in \cite{FAG}.}. We therefore briefly recall some of the basic notions, taken mainly from A.\ Vistoli's article in \cite{FAG}.
Fix a category \(\cat{C}\) and a \emph{category over} \(\cat{C}\), i.e.\ a functor \(p\co\cat{F}\ra\cat{C}\). 
To an object \(T\) in \(\cat{C}\), let \(\cat{F}(T)\); the \emph{fiber of \(\cat{F}\) over \(T\)} (or just the \emph{fibre category}), denote the subcategory of arrows \(\phi\) in \(\cat{F}\) such that \(p(\phi)=\id_{T}\).
An arrow \(\phi_{1}\co\xi\ra\xi_{1}\) in \(\cat{F}\) is \emph{cocartesian} if for any arrow \(\phi_{2}\co\xi\ra\xi_{2}\) in \(\cat{F}\) and any arrow \(f_{21}\co p(\xi_{1})\ra p(\xi_{2})\) in \(\cat{C}\) with \(f_{21} p(\phi_{1})=p(\phi_{2})\) there exists a unique arrow \(\phi_{21}\co\xi_{1}\ra\xi_{2}\) with \(p(\phi_{21})=f_{21}\) and \(\phi_{21}\phi_{1}=\phi_{2}\). If for any arrow \(f\co T\ra T'\) in \(\cat{C}\) and any object \(\xi\) in \(\cat{F}\) with \(p(\xi)=T\) there exists a cocartesian arrow \(\phi\co\xi\ra \xi'\) for some \(\xi'\) with \(p(\phi)=f\), then \(\cat{F}\) (or rather \(p\co\cat{F}\ra\cat{C}\)) is a \emph{fibred category}. Moreover, \(\xi'\) will be called a \emph{base change} of \(\xi\) by \(f\). If \(\xi''\) is another base change of \(\xi\) by \(f\) then \(\xi'\) and \(\xi''\) are isomorphic over \(T'\) by a unique isomorphism. 
A \emph{morphism of fibred categories} \((\cat{F}_{1},p_{1})\ra(\cat{F}_{2},p_{2})\) over \(\cat{C}\) is a functor \(F\co\cat{F}_{1}\ra\cat{F}_{2}\) with \(p_{2}F=p_{1}\) such that \(\phi\) cocartesian implies \(F(\phi)\) cocartesian. 
A category with all arrows being isomorphisms is a groupoid. A fibred category \(\cat{F}\) over \(\cat{C}\) is called a \emph{category fibred in groupoids} (often abbreviated to groupoid) if all fibres \(\cat{F}(T)\) are groupoids. Then all arrows in \(\cat{F}\) are cocartesian. 

Let \(\Algf\) be the category with objects faithfully flat finite type algebras \(h\co S\ra R\) of noetherian rings. Let \(h_{i}\co S_{i}\ra R_{i}\), \(i=1,2\), be two such algebras. An arrow \((g,f)\co h_{1}\ra h_{2}\) is a pair of ring homomorphisms \(g\co S_{1}\ra S_{2}\) and \(f\co R_{1}\ra R_{2}\) such that \(h_{2}g=fh_{1}\) and such that the induced map \(R_{1}\ot_{S_{1}} S_{2}\ra R_{2}\) is an isomorphism:
\begin{equation}\label{eq.coc}
\xymatrix@C-0pt@R-8pt@H-30pt{
R_{1}\ar[r]^{f} & R_{2} & R_{1}\ot_{S_{1}}S_{2} \ar[l]_(0.6){\simeq} \\
S_{1}\ar[u]^{h_{1}}\ar[r]^{g} & S_{2}\ar[u]_{h_{2}}
}
\end{equation}
Arrows are composed by composing maps `coordinate-wise'.
Let \(\cat{NR}\) denote the category of noetherian rings. The forgetful functor \(p\co \Algf\ra\cat{NR}\) which takes \((g,f)\co h_{1}\ra h_{2}\) to \(g\co S_{1}\ra S_{2}\) makes \(\Algf\) a category fibred in groupoids over \(\cat{NR}\).

Let \(\modf\) be the category of pairs \((h\co S\ra R,N)\) with \(h\) in \(\Algf\) and \(N\) an \(S\)-flat finite \(R\)-module. A morphism \((h_{1}\co S_{1}\ra R_{1},N_{1})\ra (h_{2}\co S_{2}\ra R_{2},N_{2})\) is a morphism \((g,f)\co h_{1}\ra h_{2}\) in \(\Algf\) and a \(f\)-linear map \(\alpha\co N_{1}\ra N_{2}\). Then \(\alpha\) is cocartesian with respect to the forgetful functor \(F\co \modf\ra \Algf\) if the induced map \(N_{1}\ot_{S_{1}} S_{2}\ra N_{2}\) is an isomorphism. All objects admit arbitrary base change. It follows that \(\modf\) is a fibred category and it is \emph{fibred in additive categories} over \(\Algf\) since the fibre categories for any category \(\cat{X}\) over \(\Algf\) are additive and the fibre of homomorphisms are additive groups with bilinear composition (cf.\ \cite[3.2]{ile:12a} for a precise definition).

Fix integers \(n\geq 0\) and \(r, s>0\). We define some full subcategories of \(\modf\) by properties of their objects \((h\co S\ra R,N)\) as follows.

\begin{center}
\setlength{\extrarowheight}{2,5pt}
\begin{tabularx}{\linewidth}[t]{ r | X }
\textit{Category} & \textit{Property}
\\[0.5ex]
\(\cat{P}\) & \(N\) is \(R\)-projective \\
\(\Pf(n)\) & \(N\) has an \(R\)-projective resolution of length \(\leq n\) \\
\(\cat{X}(r)\) & \(N\) is \(r\)-stably reflexive with respect to \(h\) \\ 
\(\cat{Y}(r,s)\) & \(N\) is isomorphic to a direct sum of \(R\)-modules \(\oplus N_{i}\) such that for all \(i\), \(\syz{R}{n_{i}} N_{i}\) is \((r+n_{i})\)-stably reflexive w.r.t.\ \(h\) for some \(0\leq n_{i}<s\) \\
\end{tabularx}
\end{center}
For any category \(\cat{U}\) over \(\Algf\) let \(\cat{U}_{h}\) denote the fibre category over an object \(h\co S\ra R\) in \(\Algf\), e.g.\ \(\modf_{h}\) and \(\cat{Y}(r,s)_{h}\). Note that \(\modf_{h}\) is isomorphic as category to the full subcategory of \(S\)-flat modules in \(\cat{mod}_{R}\). This identification is used without further mention.
The definitions are also meaningful in the case that \(h\) is a flat homomorphism of local noetherian rings (and \(N\) is \(S\)-flat and \(R\)-finite).
By Schanuel's lemma the definition of \(\cat{Y}(r,s)\) is not depending on the choice of \(n_{i}\)th syzygy. If \(\syz{R}{n}{N}\) is \((r+n)\)-stably reflexive over \(S\) and \(r+n>1\) then \(\syz{R}{n+1}{N}\) is \((r+n-1)\)-stably reflexive over \(S\) by Lemma \ref{lem.nstabex2}. Since \(N\) is flat Proposition \ref{prop.main} implies that all the categories are fibered in additive categories over \(\Algf\) (\(r>1\) is needed in the case of \(\cat{X}(r)\) and \(\cat{Y}(r,s)\)). 
They all contain \(\cat{P}\) as a subcategory fibred in additive categories and there are corresponding quotient categories \(\modf/\cat{P},\dots\) which all are fibred in additive categories over \(\Algf\) by \cite[3.4]{ile:12a}.
\begin{defn}\label{def.variant}
Let \(\cat{X}\) be a subcategory of a category \(\cat{Y}\). An arrow \(\pi\co M\ra N\) in \(\cat{Y}\) is called a \emph{right \(\cat{X}\)-approximation of \(N\)} if \(M\) is in \(\cat{X}\) and any \(M'\ra N\) with \(M'\) in \(\cat{X}\)  factors through \(\pi\). Dually, \(\iota\co N\ra L\) is called a \emph{left \(\cat{X}\)-approximation of \(N\)} if \(L\) is in \(\cat{X}\) and any \(N\ra L'\) with \(L'\) in \(\cat{X}\) factors through \(\iota\). The approximations need not be unique. The subcategory \(\cat{X}\) is \emph{contravariantly \textup{(}covariantly\textup{)} finite in \(\cat{Y}\)} if every object \(N\) in \(\cat{Y}\) has a right (left) \(\cat{X}\)-approximation.

An arrow \(\pi\co  M\ra N\) in \(\cat{Y}\) is called \emph{right minimal} if for any \(\eta\co M\ra M\) with \(\pi\eta=\pi\) it follows that \(\eta\) is an automorphism. Dually, \(\pi\) is called \emph{left minimal} if for any \(\theta\co N\ra N\) with \(\theta\pi=\pi\) it follows that \(\theta\) is an automorphism. 
\end{defn}
We will simply call a right (left) \(\cat{X}\)-approximation for minimal if it is right (left) minimal. Note that if \(\pi\co M\ra N\) and \(\pi'\co  M'\ra N\) both are minimal right \(\cat{X}\)-approximations then there exists an isomorphism \(\phi\co M\ra M'\) with \(\pi=\pi'\phi\), and similarly for minimal left \(\cat{X}\)-approximation morphisms. 

In the unpublished manuscript \cite{buc:86} R.-O.\ Buchweitz proved an approximation theorem for (not necessarily commutative) Gorenstein rings by applying a resolution of a complex. This is also the basic construction in the proof of the following result.
\begin{thm}\label{thm.approx1}
Given an \(h\co S\ra R\) in \(\Algf\) and suppose \(r, s>1\)\textup{.} 
\begin{enumerate}
\item[(i)] For any module \(N\) in \(\cat{Y}(r,s)_{h}\) there are short exact sequences
\begin{enumerate}\label{eq.rapprox}
\item \(0\ra L\lra M\lra N\ra 0\) \,\,with \(L\) in \(\Pf(s-2)_{h}\) and \(M\) in \(\cat{X}(r)_{h}\)
\item \(0\ra N\lra L'\lra M'\ra 0\)  \,\,with \(L'\) in \(\Pf(s-1)_{h}\) and \(M'\) in \(\cat{X}(r-1)_{h}\)
\end{enumerate}
which are preserved by any base change in \(\cat{NR}\)\textup{.} 

\item[(ii)] If \(s\leq r\) then \textup{(a)} is a right \(\cat{X}(r)\)-approximation of \(N\) and if \(s\leq r-2\) then \textup{(b)} is a left \(\Pf(s-1)\)-approximation of \(N\)\textup{.} 

\item[(iii)] In particular \(\cat{X}(r)\) is contravariantly finite in \(\cat{Y}(r,r)\) and \(\Pf(r-3)\) is covariantly finite in \(\cat{Y}(r,r-2)\)\textup{.}

\item[(iv)] Furthermore\textup{,} if \(h\) is local then there exist minimal approximations \textup{(a)} and \textup{(b).}
\end{enumerate}
\end{thm}
\begin{proof}
A direct sum of short exact sequences as in (a) (or in (b)) gives a new short exact sequence of the same kind. All statements in Theorem \ref{thm.approx1} are independent of the \(n_{i}\) appearing in the definition of \(\cat{Y}(r,s)\). We therefore
assume that \(\syz{R}{n}N\) is \(m\)-stably reflexive with respect to \(h\) where \(m=r+n\) and \(n<s\). We also assume that \(n>0\) and leave the case \(n=0\) to the reader.

(i) Let \(P\thr N\) be an \(R\)-projective resolution of \(N\). Recall our convention \((P^{\vee})^{i+1}=(P^{-i})^{*}\). By Proposition \ref{prop.main} (i) the dual complex \(P^{\vee}\) is acyclic in degrees between \(n+2\) and \(n+m\). Choose an \(R\)-projective resolution \(Q\ra \tau^{\leq n+1}P^{\vee}\) of the soft truncation: 
\begin{equation}
\xymatrix@C-0pt@R-8pt@H-30pt{
\dots & 0\ar[l]\ar[d] & Q^{n+1}\ar[l]\ar[d] & Q^{n}\ar[l]\ar[d] & Q^{n-1}\ar[l]\ar[d] & \dots\ar[l]  \\
\dots & 0\ar[l] & \ker d_{P^{\vee}}^{n+1}\ar[l] & (P^{\vee})^{n}\ar[l]_(0.4){d_{P^{\vee}}^{n}} & (P^{\vee})^{n-1}\ar[l]_{d_{P^{\vee}}^{n-1}} & \dots\ar[l]
}
\end{equation}
Let \(C\) denote the mapping cone \(C(f)\) of the corresponding map \(f\co Q\ra P^{\vee}\), in particular \(C^{i}=Q^{i+1}{\oplus}(P^{\vee})^{i}\) for all \(i\). The short exact sequence of complexes 
\begin{equation}\label{eq.C}
0\ra P^{\vee}\lra C\lra Q[1]\ra0
\end{equation}
gives a long exact sequence in cohomology. Since \(Q^{n+2}=0\) and \(\cH^{i}(Q)\cong\cH^{i}(P^{\vee})\) for all \(i\leq n+1\), \(\cH^{i}(C)=0\) for all \(i\leq n+1\) and \(\dots\ra C^{n-1}\ra C^{n}\ra \ker d_{C}^{n+1}\ra 0\) is exact. Furthermore, the natural map \(\ker d_{P^{\vee}}^{n+1}\ra \ker d_{C}^{n+1}\) is an isomorphism and \(\ker d_{P^{\vee}}^{n+1}\cong (\syz{R}{n}N)^{*}\). By Proposition \ref{prop.main}, \((\syz{R}{n}N)^{*}\) is \(m\)-stably reflexive as \(R\)-module, in particular \(C^{\vee}\) is acyclic in degrees between \(-n+2\) and \(r\). Since \((C^{\vee})^{i}=P^{i}\) for \(i\leq -n\) and the dual of \eqref{eq.C} gives an exact sequence \(\cH^{-n}(Q[1]^{\vee})\ra \cH^{-n}(C^{\vee})\ra \cH^{-n}(P)\) with the first and the third term equal to \(0\), we get \(\cH^{i}(C^{\vee})=0\) for \(i\leq -n\). The sequences \(0\ra (\syz{R}{n}N)^{**}\ra (C^{n})^{*}\ra (C^{n-1})^{*}\) and \(P^{-n-1}\ra P^{-n}\ra \syz{R}{n}N\ra 0\) are exact, \(P^{-i}=(C^{i+1})^{*}\) for \(i\leq -n\), and \(\syz{R}{n}N\cong (\syz{R}{n}N)^{**}\) by assumption, hence \((C^{n+2})^{*}\ra (C^{n+1})^{*}\ra (\syz{R}{n}N)^{**}\ra 0\) is exact and \(\cH^{-n+1}(C^{\vee})=0\). Together this gives \(\cH^{i}(C^{\vee})=0\) for \(i\leq r\). The long exact sequence \(\dots\ra \cH^{i}(P)\ra\cH^{i}(Q^{\vee})\ra\cH^{i+1}(C^{\vee})\ra\dots\) of the dual of \eqref{eq.C} then implies that \(\cH^{i}(Q^{\vee})=0\) for \(i<0\).
We obtain the following defining diagram
\begin{equation}
\xymatrix@C-0pt@R-8pt@H-30pt{
\dots \ar[r]^(0.45){d^{-2}} & (Q^{2})^{*} \ar[rr]^{d^{-1}} \ar[d] && (Q^{1})^{*} \ar[rr]^{d^{0}} \ar[dr] && (Q^{0})^{*} \ar[r]^(0.6){d^{1}} & \dots  \\
& L \ar[r] & M \ar[ur] \ar[dr] && L' \ar[r] & M' \ar[u] \\
&&& N \ar[ur]
}
\end{equation}
with \(L:=\coker d^{-2}\), \(M:=\ker d^{0}\), \(L':=\coker d^{-1}\) and \(M':=\ker d^{1}\).
Since \(Q[1]^{\vee}\) equals \(C^{\vee}\) in degrees \(\geq 1\) we have that \(M\) is the \((r+1)\)-syzygy of \(K:=\coker\{(C^{-r+1})^{*}\ra (C^{-r})^{*}\}\). By assumption and Proposition \ref{prop.main} (ii), \(\syz{R}{n}N\) and \((\syz{R}{n}N)^{*}\) are \(S\)-flat. Let \(S\ra S'\) be a ring homomorphism and let \(N'=N\ot_{S} S'\) denote the induced \(R'=R\ot_{S}S'\)-module. By Proposition \ref{prop.main}, \(\syz{R}{n}N'\cong(\syz{R}{n}N)\ot_{S}S'\) is \(m\)-stably reflexive and so is \((\syz{R}{n}N')^{*}\). It follows that the base change \(C'=C\ot_{S}S'\) and \((C')^{\vee}\cong (C^{\vee})'\) retains the properties of \(C\) and \(C^{\vee}\). In particular a truncation of \((C')^{\vee}\) gives an \(R'\)-projective resolution of \(K'=K\ot_{S}S'\). By Proposition \ref{prop.nakayama} (i) this implies that \(K\) is \(S\)-flat. Moreover, since \(\xt{i}{R'}{K'}{R'}\cong\cH^{i-r}(C')\) and \(\cH^{i-r}(C')=0\) for \(0<i\leq 2(r+n)\), \(K\) is \(2r\)-orthogonal to \(h\). By Proposition \ref{prop.ortho} (i), \(M\) is \(r\)-stably reflexive with respect to \(h\). This argument also gives that \(M'\cong \syz{R}{r}K\) is \((r-1)\)-stably reflexive with respect to \(h\). We note that \(L'\) and \(L\) are \(S\)-flat with finite projective resolutions given by the truncations of \(Q^{\vee}\) of length \(n\) and \(n-1\), respectively.

(ii) We first prove the statement for the fibre categories \(\cat{X}(r)_{h}\) and \(\Pf(s-1)_{h}\). Consider the sequence in (a), an \(R\)-module \(M_{1}\) in \(\cat{X}(r)_{h}\) and an \(R\)-linear map \(\pi\co M_{1}\ra N\). By Proposition \ref{prop.main} (i), \(M_{1}\) is \(r\)-stably reflexive as \(R\)-module. The map \(\pi\) lifts to a map \(M_{1}\ra M\) if \(s\leq r\) since this implies \(r-\pdim L\geq 2\) and by Lemma \ref{lem.approx1} below \(\xt{1}{R}{M_{1}}{L}=0\). Consider the sequence in (b), an \(R\)-module \(L_{1}\) in \(\Pf(s-1)_{h}\) and an \(R\)-linear map \(\iota\co N\ra L_{1}\). By Proposition \ref{prop.main} (i), \(M'\) is \((r-1)\)-stably reflexive as \(R\)-module. The map \(\iota\) extends to a map \(L'\ra N\) if \(s\leq r-2\) since then \(r-1-\pdim L_{1}\geq r-s\geq 2\) and by Lemma \ref{lem.approx1}, \(\xt{1}{R}{M'}{L_{1}}=0\).

For the general case, consider an object \((h_{1}\co S_{1}\ra R_{1},M_{1})\) in \(\cat{X}(r)\) and a morphism \((h_{1}\co S_{1}\ra R_{1},M_{1})\ra (h\co S\ra R,N)\). I.e.\ there is a cocartesian square of ring homomorphisms as in \eqref{eq.coc} and an \((R_{1}\ra R)\)-linear map \(\phi\co M_{1}\ra N\). Let \(M^{\#}_{1}\) denote \(R\ot_{R_{1}}M_{1}\), which is an object in \(\cat{X}(r)_{h}\). There is a unique induced \(R\)-linear map \(\pi\co M_{1}^{\#}\ra N\). By the preceding argument there is a lifting of \(\pi\) to a map \(\tilde{\pi}\co M_{1}^{\#}\ra M\). The composition of the base change map \(M_{1}\ra M_{1}^{\#}\) with \(\tilde{\pi}\) gives the lifting of \(\phi\). The argument for the general \(\Pf(s-1)\)-case is similar.

(iv) In the local case we choose \(Q\) to be a minimal complex of free modules. Then (a) and (b) are minimal approximations by a result analogous to \cite[6.2]{ile:12a}.
\end{proof}
\begin{lem}\label{lem.approx1}
Let \(R\) be a coherent ring and fix a positive integer \(r\). Assume \(L_{1}\) and \(M_{1}\) are coherent \(R\)-modules with \(L_{1}\) of finite projective dimension and \(M_{1}\) \(r\)-stably reflexive as \(R\)-module. Then
\begin{equation*}
\xt{i}{R}{M_{1}}{L_{1}}=0\, \text{ for }\, 0<i<r-\pdim L_{1}.  
\end{equation*}
\end{lem}
\begin{proof}
By induction on \(d=\pdim L_{1}\). If \(d=0\), \(L_{1}\) is direct summand of a free finite rank \(R\)-module and the definition of \(r\)-stably reflexive gives the vanishing. Assume \(d>0\). There is a short exact sequence \(0\ra L_{2}\ra P\ra L_{1}\ra 0\) such that \(P\) is finite and projective, and \(\pdim L_{2}=d-1\). Apply \(\hm{}{R}{M_{1}}{-}\) and inspect the long exact sequence.
\end{proof}
Morphisms \(\cat{A}_{1}\ra\cat{A}_{2}\ra\cat{A}_{3}\) of categories fibred in additive categories over some base category is a \emph{short exact sequence} if \(\cat{A}_{1}\ra\cat{A}_{2}\) is an inclusion and \(\cat{A}_{2}\ra \cat{A}_{3}\) is equivalent to the quotient morphism \(\cat{A}_{2}\ra\cat{A}_{2}/\cat{A}_{1}\).
\begin{thm}\label{thm.approx2} Suppose \(r>2\) and \(s>1\)\textup{.} 
\begin{enumerate}
\item[(i)] For any \(s\leq r\) the right \(\cat{X}(r)\)-approximation in \textup{Theorem \ref{thm.approx1}} induces a morphism \(j^{!}\co \cat{Y}(r,s)/\cat{P}\ra\cat{X}(r)/\cat{P}\) of categories fibred in additive categories which is a right adjoint to the full and faithful inclusion morphism \(j_{!}\co \cat{X}(r)/\cat{P}\linebreak[1]\ra\cat{Y}(r,s)/\cat{P}\)\textup{.}
\item[(ii)] For any \(s\leq r-2\) the left \(\Pf(s-1)\)-approximation in \textup{Theorem \ref{thm.approx1}} induces a morphism \(i^{*}\co \cat{Y}(r,s)/\cat{P}\ra\Pf(s-1)/\cat{P}\) of categories fibred in additive categories which is a left adjoint to the full and faithful inclusion morphism \(i_{*}\co \Pf(s-1)/\cat{P}\ra \cat{Y}(r,s)/\cat{P}\)\textup{.}
\item[(iii)] Together these maps give the following commutative diagram of short exact sequences of categories fibred in additive categories\textup{:}
\begin{equation*}
\xymatrix@C-0pt@R-8pt@H-30pt{
0\ar[r] & \, \Pf(r-3)/\cat{P} \ar[r]^(0.47){i_{*}}\ar[d]_(0.4){\id} & \cat{Y}(r,r-2)/\cat{P}\ar[r]^(0.57){j^{!}}\ar@{=}[d] & \cat{X}(r)/\cat{P} \ar[r] & 0 \\
0 & \Pf(r-3)/\cat{P}\ar[l] & \cat{Y}(r,r-2)/\cat{P}\ar[l]_(0.47){i^{*}} & \,  \cat{X}(r)/\cat{P}  \ar[l]_(0.37){j_{!}}\ar[u]_{\id} & 0 \ar[l]
}
\end{equation*}
\end{enumerate}
\end{thm}
\begin{proof}
The proof very much resembles the proof of \cite[4.5]{ile:12a}. By 
base change the arguments can be reduced to statements about the fibre categories above an object \(h\co S\ra R\) in the base category \(\Algf\). In the fibre categories there are many similar arguments, we only give some of them. See the proof of \cite[4.5]{ile:12a} for some of the cases left out here.

Recall that the objects in \(\modf\) are pairs \((h,N)\) where \(N\) is an \(R\)-module. The functors in Theorem \ref{thm.approx2} act as the identity in the first coordinate.  
For every module \(N_{v}\) in \(\cat{Y}(r,s)_{h}\) we fix a right \(\cat{X}(r)\)-approximation and a left \(\Pf(s-1)\)-approximation 
\begin{equation}\label{eq.nu}
0\ra L_{v}\lra  M_{v}\xra{\pi_{v}} N_{v}\ra 0\quad \text{ and }\quad 0\ra N_{v}\xra{\iota_{v}} L'_{v}\lra M'_{v}\ra 0
\end{equation}
assured in Theorem \ref{thm.approx1} (i)
with \(\pi_{v}=\id\) if \(N_{v}\) is contained in the subcategory \(\cat{X}(r)_{h}\) and \(\iota_{v}=\id\) if \(N_{v}\) is contained in the subcategory \(\Pf(s-1)_{h}\).
Define the functor \(j^{!}\) on objects by \(j^{!}((h,N_{v}))=(h,M_{v})\). 
Given objects \((h_{1}\co S_{1}\ra R_{1},N_{v_{1}})\) and \((h_{2}\co S_{2}\ra R_{2},N_{v_{2}})\) in \(\cat{Y}(r,s)\), a morphism 
\begin{equation}\label{eq.map}
((f,g),\phi)\co(h_{1}\co S_{1}\ra R_{1},N_{v_{1}})\ra (h_{2}\co S_{2}\ra R_{2},N_{v_{2}})
\end{equation}
is a cocartesian square of ring homomorphisms as in \eqref{eq.coc} and an \(f\)-linear map \(\phi\co N_{v_{1}}\ra N_{v_{2}}\). For every morphism \(((f,g),\phi)\); cf.\ \eqref{eq.map}, choose a map \(\psi\co M_{v_{1}}\ra M_{v_{2}}\) which lifts \(\phi\). The existence of \(\psi\) is assured by Theorem \ref{thm.approx1} (ii). Define \(j^{!}(((f,g),[\phi]))=((f,g),[\psi])\) where brackets denote the image in the quotient category.

To prove that the functor \(j^{!}\) commutes with composition consider any maps \(\phi_{21}\co N_{v_{1}}\ra N_{v_{2}}\) and \(\phi_{32}\co N_{v_{2}}\ra N_{v_{3}}\) (without making the homomorphisms of algebras explicit). Above we have chosen three maps; \(\psi_{21}\co M_{v_{1}}\ra M_{v_{2}}\), \(\psi_{32}\co M_{v_{2}}\ra M_{v_{3}}\) and \(\psi_{31}\co M_{v_{1}}\ra M_{v_{3}}\), which lifts \(\phi_{21}\), \(\phi_{32}\) and \(\phi_{32}\phi_{21}\), respectively. Then \(\psi_{32}\psi_{21}-\psi_{31}\) factors through a map \(\delta\co M_{v_{1}}\ra L_{v_{3}}\); cf.\ \eqref{eq.nu}. By the cocartesian property of base change one may assume that the maps are in the same fibre category. Then \(L_{v_{3}}\) sits in a short exact sequence \(0\ra L''\ra P\ra L_{v_{3}}\ra 0\) with \(P\) projective and \(\pdim L''\leq s-3\) by Theorem \ref{thm.approx1} (i). By Lemma \ref{lem.approx1} and Proposition \ref{prop.main} (i), \(\xt{1}{}{M_{1}}{L''}=0\) since \(r-(s-3)>1\). It implies that \(\delta\) factors through \(P\) and \([\psi_{32}\psi_{21}]=[\psi_{31}]\). In a similar way one proves that \(j^{!}\) is well defined. 

To show that \(j^{!}\) preserves cocartesian morphisms, consider a cocartesian morphism \(\Phi=((f,g),[\phi])\co(h_{1},N_{v_{1}})\ra (h_{2},N_{v_{2}})\); cf.\ \eqref{eq.map}, in \(\cat{Y}(r,s)/\cat{P}\). Assume \(j^{!}(\Phi)=\Psi\) where \(\Psi=((f,g),[\psi])\co (h_{1},M_{v_{1}})\ra (h_{2},M_{v_{2}})\). The base change \(R_{2}\ot_{R_{1}}N_{v_{1}}\) is an object in \(\cat{Y}(r,s)_{h_{2}}\). The induced map \([\id\ot\phi]\co R_{2}\ot_{R_{1}}N_{v_{1}}\ra N_{v_{2}}\) is an isomorphism in \((\cat{Y}(r,s)/\cat{P})_{h_{2}}\) by \cite[3.5]{ile:12a}. Since \(j^{!}\) is a well-defined functor, \([\id\ot\psi]\co R_{2}\ot_{R_{1}}M_{v_{1}}\ra M_{v_{2}}\), which lifts \([\id\ot\phi]\), is an isomorphism in \((\cat{X}(r)/\cat{P})_{h_{2}}\). Composing the base change map \(M_{v_{1}}\ra R_{2}\ot_{R_{1}}M_{v_{1}}\) with \(\id\ot\psi\) gives \(\psi\). Hence \(\Psi\) is a cocartesian morphism by \cite[3.5]{ile:12a}.

To prove that \((j_{!},j^{!})\) is a pair of adjoint functors, let \(\pi_{v}\co M_{v}\ra N_{v}\) be one of the fixed \(\cat{X}(r)_{h}\)-approximations; cf.\ \eqref{eq.nu}. Let \((h_{1}:S_{1}\ra R_{1},M_{1})\) be an object in \(\cat{X}(r)\) and \(((f,g),\phi_{1})\co (h_{1},M_{1})\ra (h,N_{v})\) a morphism. The induced map \(\phi\co M:=R\ot_{R_{1}}M_{1}\ra N_{v}\) lifts to a map \(\psi\co M\ra M_{v}\) since \(\xt{1}{}{M}{L_{v}}=0\) (cf.\ \eqref{eq.nu}) by Lemma \ref{lem.approx1} as \(r-(s-2)>1\). Composing \(M_{1}\ra M\) with \(\psi\) gives a lifting of \(\phi_{1}\). This shows surjectivity of the adjointness map 
\begin{equation}
(\pi_{v})_{*}\co \hm{}{}{M_{1}}{j^{!}N_{v}}\lra \hm{}{}{j_{!}M_{1}}{N_{v}}\,.
\end{equation}
Injectivity of \((\pi_{v})_{*}\) is similar to the proof above that \(j^{!}\) commutes with composition and reduces to the same inequality \(r-(s-3)>1\); see the proof of \cite[4.5]{ile:12a}.

We show exactness in the lower row of the diagram. Let \(\Phi=((f,g),\phi)\) be a map in \(\cat{Y}(r,r-2)\); cf.\ \eqref{eq.map}. Then \(i^{*}((f,g),[\phi])=((f,g),[\lambda])\) where \(\lambda\co L_{v_{1}}'\ra L_{v_{2}}'\) is a chosen, fixed extension of \(\phi\); cf.\ \eqref{eq.nu}. Assume \([\lambda]=0\), i.e.\ \(\lambda\) factors through a projective module. We would like to show that \(\Phi\) factors through an object in \(\cat{X}(r)\) and it is sufficient to show this in the case \(\phi\) and hence \(\lambda\) are in a fibre category (over an algebra \(h\)), i.e.\ when all the modules are defined over the same ring. Then naturality of long-exact sequences obtained from \eqref{eq.nu} gives a commutative diagram:
\begin{equation}\label{eq.star}
\xymatrix@C-0pt@R-8pt@H-30pt{
\xt{1}{}{L'_{v_{2}}}{L_{v_{2}}} \ar[r]^{\lambda^{*}}\ar[d] & \xt{1}{}{L'_{v_{1}}}{L_{v_{2}}} \ar[d] \\
\xt{1}{}{N_{v_{2}}}{L_{v_{2}}} \ar[r]^{\phi^{*}} & \xt{1}{}{N_{v_{1}}}{L_{v_{2}}}
}
\end{equation}
Now \(M'_{v_{j}}\) is contained in \(\cat{X}(r-1)_{h}\) for \(j=1,2\) and \(L_{v_{2}}\) is contained in \(\Pf(r-4)_{h}\) so \(\xt{i}{}{M '_{v_{j}}}{L_{v_{2}}}=0\) for \(i=1,2\) by Lemma \ref{lem.approx1}. Hence the vertical maps in \eqref{eq.star} are isomorphisms. As \(\lambda\) factors through a projective module, \(\lambda^{*}=0\) in \eqref{eq.star} and so \(\phi^{*}=0\). Let \(e_{2}\) denote the short exact sequence \(0\ra L_{v_{2}}\ra M_{v_{2}}\ra N_{v_{2}}\ra 0\). Since \(\phi^{*}=0\) the short exact sequence \(\phi^{*}e_{2}\) splits and hence \(\phi\co N_{v_{1}}\ra N_{v_{2}}\) factors through \(M_{v_{2}}\) in \(\cat{X}(r)_{h}\). 

The other parts of the proof are similar.
\end{proof}
Finally we give the relative Cohen-Macaulay approximation result which is the `limit' of Theorems \ref{thm.approx1} and \ref{thm.approx2}. Define full subcategories of \(\modf\) by objects \((h\co S\ra R,N)\) as follows.

\begin{center}
\setlength{\extrarowheight}{2,5pt}
\begin{tabularx}{\linewidth}[t]{ r | X }
\textit{Category} & \textit{Property}
\\[0.5ex]
\(\Pf\) & \(N\) has a finite \(R\)-projective resolution \\
\(\cat{X}\) & \(N\) is stably reflexive with respect to \(h\) \\
\(\cat{Y}\) & \((h,\syz{R}{n} N)\) is in \(\cat{X}\) for an integer \(n\geq 0\) depending on \(N\) 
\end{tabularx}
\end{center}
We note that \(\cat{X}\), \(\cat{Y}\) and \(\Pf\) are fibred in additive categories over \(\Algf\). The expression `without the parametres' in the following theorem is short for `after removing the parametres from the symbols in the statements (e.g.\ \(\cat{Y}(r,s)_{h}\) becomes \(\cat{Y}_{h}\)) and removing the conditions involving parametres'.
\begin{thm}\label{thm.approx3}
The statements in \textup{Theorem \ref{thm.approx1}} and \textup{Theorem \ref{thm.approx2}} without the parametres are true\textup{.}
\end{thm}
\begin{proof}
One checks that all four of the Auslander-Buchweitz axioms (cf.\ \cite{ile:12a}) hold in the fibre categories. The main ingredient here is Proposition \ref{prop.main}. One also checks that the two base change axioms in \cite{ile:12a} hold (formally one also has to consider the category \(\cat{mod}\) fibred in abelian categories over \(\Algf\) of pairs \((h,N)\) where \(N\) is not necessarily \(S\)-flat). Then this is a corollary of \cite[4.4 and 4.5]{ile:12a}. Alternatively one can follow the proofs of Theorems \ref{thm.approx1} and \ref{thm.approx2}.
\end{proof}
If \(\cat{C}\) is an additive subcategory of a module category then \(\hat{\cat{C}}\) denotes the full subcategory of modules \(N\) which have a finite \(\cat{C}\)-resolution \(0\ra C^{-n}\ra\dots\ra C^{0}\ra N\ra 0\). The minimal \(n\) is the \emph{\(\cat{C}\)-resolving dimension} \(\resdim{\cat{C}}{N}\) of \(N\). Let \(A\) be a coherent ring and denote by \(\cat{X}_{A}\) the category of stably reflexive \(A\)-modules and by \(\cat{Y}_{A}\) the category of coherent \(A\)-modules \(N\) with \(\syz{A}{n}N\) in \(\cat{X}_{A}\) for some \(n\). If \(A\) is noetherian and \(N\) a finite \(A\)-module, \(\resdim{\cat{X}_{A}}{N}\) equals the \emph{Gorenstein dimension} of \(N\) defined by Auslander and Bridger and \(\cat{Y}_{A}=\hat{\cat{X}}_{A}\); see \cite[3.13]{aus/bri:69}. This extends to the relative setting and gives a characterisation of \(\cat{Y}\).
\begin{lem}\label{lem.Gordim}
Let \(h\co S\ra R\) be a ring homomorphism in \(\Algf\) and \(N\) a module in \(\modf_{h}\)\textup{.} Put \(Z=\im\{\mSpec R\ra\Spec S\}\)\textup{.} The following are equivalent\textup{:}
\begin{enumerate}
\item[(i)] \(\syz{R}{n}N\) is in \(\cat{X}_{h}\)\textup{.}
\item[(i')] \(\syz{R_{s}}{n}N_{s}\) is in \(\cat{X}_{R_{s}}\) for all \(s\in Z\)\textup{.}
\item[(ii)] \(\resdim{\cat{X}_{h}}{N}\leq n\)\textup{.}
\item[(ii')] \(\resdim{\cat{X}_{R_{s}}}{N_{s}}\leq n\) for all \(s\in Z\)\textup{.}
\end{enumerate}
\end{lem}
\begin{proof}
Put \(M=\syz{R}{n}N\). Since \(N\) is \(S\)-flat \(M\) is \(S\)-flat too and \(M_{s}\) is (stably) isomorphic to \(\syz{R_{s}}{n}N_{s}\). Hence (i')\(\lRa\)(i). Since \(\cat{P}_{h}\sbeq\cat{X}_{h}\), (i) implies (ii). The fibre at \(s\) of an \(\cat{X}_{h}\)-resolution of \(N\) gives an \(\cat{X}_{R_{s}}\)-resolution of \(N_{s}\) hence (ii) implies (ii'). Finally (ii')\(\lRa\)(i') by \cite[3.13]{aus/bri:69}.
\end{proof}
To shed some further light on \(\cat{Y}\) we give the following perhaps not so well known results of Auslander and Bridger (who attribute (i) to C.\ Peskine and L.\ Szpiro).
\begin{prop}[{\cite[4.12, 13, 35]{aus/bri:69}}]\label{prop.ausbri}
Let \(A\) be a noetherian local ring and \(N\) a finite \(A\)-module\textup{.}
\begin{enumerate}
\item[(i)] Suppose \((f)=(f_{1},\dots,f_{n})\) is an \(A\)-regular sequence which annihilates \(N\)\textup{.} Then \(\resdim{\cat{X}_{A}}{N}=\grade_{A} N=n\) if and only if \(N\) is in \(\cat{X}_{A/(f)}\)\textup{.}

\item[(ii)] If \(N\) is in \(\cat{\hat{X}}_{A}\) then \(\resdim{\cat{X}_{A}}{N}+\depth N=\depth A\)\textup{.}

\item[(iii)] If \(N\) is in \(\cat{\hat{X}}_{A}\) then \(\resdim{\cat{X}_{A}}{N}=\min\{n\,\vert\, \xt{i}{A}{N}{A}=0\text{ for all }i>n\}\)\textup{.}
\end{enumerate}
\end{prop}
These results have recently been generalised to coherent rings; see \cite{hum/mar:09}.
\section{Pointed Gorenstein singularities and Knudsen's lemma }
We state and prove our version of Knudsen's lemma. We also give a general result about versal families for deformations of pointed algebras and make it explicit for isolated complete intersection singularities. Finally we generalise Knudsen's stabilisation to pointed plane curve singularities.
\begin{thm}\label{prop.Finn}
Let \(h\co S\ra R\) be a flat homomorphism of local\textup{,} noetherian rings and let \(k\) denote the residue field \(S/\fr{m}_{S}\) and \(A\) the central fibre \(R\ot_{S}k\)\textup{.} Given an \(S\)-algebra map \(R\ra S\)\textup{,} let \(I\) denote the kernel and \(I^{*}=\hm{}{R}{I}{R}\)\textup{.} Assume \(A\) is Gorenstein of dimension \(1\)\textup{.}
\begin{enumerate}
\item[(i)] The \(R\)-module \(I\) is stably reflexive with respect to \(h\)\textup{.}
\item[(ii)] The \(R\)-module \(I^{*}/R\) is isomorphic to \(S\)\textup{.}
\item[(iii)] The \(R\)-submodule \(J\) of the total quotient ring \(K(R)\) consisting of elements that multiply \(I\) into \(R\) is isomorphic to \(I^{*}\)\textup{.}
\item[(iv)] The image of the pairing \(\fr{m}_{A}\ot_{A}\fr{m}_{A}^{*}\ra A\) equals \(A\) if \(A\) is a regular ring and \(\fr{m}_{A}\) if not\textup{.} In particular \(\fr{m}_{A}^{*}\ot_{A}k\) is a \(k\)-vector space of dimension \(1\) in the regular case and dimension \(2\) otherwise\textup{.}
\end{enumerate}
\end{thm}
\begin{proof}
(i) From the short exact sequence of \(R\)-modules \(0\ra I\ra R\ra S\ra 0\) we obtain that \(I\ot_{S}k\cong\fr{m}_{A}\) and that \(I\) is \(S\)-flat. Since \(A\) is Cohen-Macaulay and \(1\)-dimensional \(\fr{m}_{A}\) is a maximal Cohen-Macaulay \(A\)-module. Since \(A\) furthermore is Gorenstein, \(I\) is stably reflexive with respect to \(h\) by Lemma \ref{lem.Gor}.

(ii) Since \(R\) is \(S\)-flat \(\{\xt{q}{R}{S}{R\ot_{S}-}\}\) is a \(\delta\)-functor. Proposition \ref{prop.nakayama} (i) implies that \(\hm{}{R}{S}{R}\linebreak[1]=0\) since \(\hm{}{A}{k}{A}=0\). So applying \(\hm{}{R}{-}{R}\) to \(0\ra I\ra R\ra S\ra 0\) gives the short exact sequence \(0\ra R\ra I^{*}\ra \xt{1}{R}{S}{R}\ra 0\). Since \(A\) is \(1\)-dimensional Gorenstein we have 
\begin{equation}\label{eq.iso}
\hm{}{A}{\fr{m}_{A}}{A}/A\cong\xt{1}{A}{A/\fr{m}_{A}}{A}\cong A/\fr{m}_{A}\cong k
\end{equation}
and \(\xt{i}{A}{k}{A}=0\) for all \(i\neq 1\). By Corollary \ref{cor.xtdef} we have that \(\xt{1}{R}{S}{R}\cong I^{*}/R\) is an \(S\)-flat \(R\)-deformation of the \(A\)-module \(k\) to \(S\). Pick an element in \(I^{*}/R\) which lifts \(1\) in \((I^{*}/R)\ot_{S}k\cong k\). Multiplication with this element defines an \(R\)-linear map \(S\ra I^{*}/R\) since \(I\) is contained in the annihilator of \(I^{*}/R\). By Nakayama's lemma the map is surjective. Since \(I^{*}/R\) is \(S\)-flat it is injective too. 

(iii) Consider the map \(m\co J\ra I^{*}\) defined by the multiplication map \(m_{\gamma}(u)=\gamma{\cdot} u\) for \(\gamma\in J\) and \(u\in I\). First we show that \(m\) is injective. Suppose there is a \(\gamma\neq 0\) such that \(\gamma{\cdot} I=0\). We have \(\gamma=ab^{-1}\) for some \(a, b\in R\) and \(I\) is annihilated by \(a\). The base change map \(\hm{}{R}{R/I}{R}\ot_{S}k\ra\hm{}{A}{k}{A}=0\) is an isomorphism by Proposition \ref{prop.nakayama} (i). This is a contradiction and \(m\) is injective.

Consider the commutative diagram 
\begin{equation}\label{eq.Finn}
\xymatrix@-0pt@C-0pt@R-8pt@H-0pt{
0\ar[r] & R \ar[r]\ar@{=}[d] & J\ar[r]\ar[d]^(0.4){m} & J/R\ar[r]\ar[d] & 0 \\
0\ar[r] & R \ar[r] & I^{*}\ar[r] & I^{*}/R\ar[r] & 0
}
\end{equation}
To show surjectivity of \(m\) we produce an element \(\tilde{\vare}\) in \(J\) which maps to a generator for \(I^{*}/R\cong S\). Consider diagram \eqref{eq.Finn} for \(R=A\). There is a non-zero divisor \(x\in \fr{m}_{A}\). The \(0\)-dimensional quotient ring \(A/(x)\) is Gorenstein, i.e.\ the socle has length \(1\) generated by an \(\bar{f}\) induced by some \(f\in A\). Put \(\vare=fx^{-1}\). Since \(\vare\notin A\), we have that \(m_{\vare}\) maps to a generator for \(\hm{}{A}{\fr{m}_{A}}{A}/A\). By (i) and Proposition \ref{prop.main} (iii), \(I^{*}\ra \hm{}{A}{\fr{m}_{A}}{A}\) is surjective. Hence \(m_{\vare}\) lifts to an \(R\)-linear map \(\psi\co I\ra R\). Pick a lifting \(\tilde{x}\) in \(I\) of \(x\). It is not a zero divisor by Corollary \ref{cor.nak}. Let \(\tilde{f}=\psi(\tilde{x})\) and put \(\tilde{\vare}=\tilde{f}{\tilde{x}}^{-1}\). Then \(m_{\tilde{\vare}}\in\hm{}{R}{I}{K(R)}\). To show that \(m_{\tilde{\vare}}\) equals \(\psi\) let \(\phi\) denote the difference \(\psi-m_{\tilde{\vare}}\in\hm{}{R}{I}{K(R)}\). If \(u\in I\), then \(\tilde{x}{\cdot}\phi(u)=\phi(\tilde{x}u)=u{\cdot}\phi(\tilde{x})=0\). But \(\tilde{x}\) is a unit in \(K(R)\), hence \(\phi=0\).

(iv) Let \(\mu\co \fr{m}_{A}\ot_{A}\fr{m}_{A}^{*}\ra A\) denote the pairing. We have already shown that \(\im\mu=\fr{m}_{A}+\vare{\cdot}\fr{m}_{A}\sbeq A\). If \(A\) is regular then \(\fr{m}_{A}\) is a principal ideal and \(\im\mu=A\). For the converse suppose \(\im\mu=A\), i.e.\ there is an element \(u\in\fr{m}_{A}\) with \(\vare u=1\). In the proof of (iii) we only assumed that \(x\) in \(\vare=fx^{-1}\) was not a zero divisor. But we can also assume that \(x\in\fr{m}_{A}\setminus\fr{m}_{A}^{2}\). Since \(fu=x\) we have \(f\notin \fr{m}_{A}\) and \(\Soc A/(x)=(\bar{f})=A/(x)\), i.e.\ \(A/(x)\) is a field, \((x)\) is the maximal ideal of \(A\), and \(A\) is regular. Hence if \(A\) is not regular then \(\im\mu=\fr{m}_{A}\).  

For the last part note that the pairing \(\mu\) composed with the inclusion \(j:A\ra \fr{m}_{A}^{*}\) equals the multiplication map \(u\ot\rho\mapsto u{\cdot}\rho\). Hence \(j\mu\) is obtained by applying \(-\ot_{A}\fr{m}_{A}^{*}\) to the inclusion \(\fr{m}_{A}\sbeq A\). By the snake lemma the following commutative diagram with exact rows
\begin{equation}\label{eq.Fin}
\xymatrix@-0pt@C-0pt@R-8pt@H-0pt{
0\ar[r] & \fr{m}_{A}\ar[d]\ar[r] & \fr{m}_{A}\ot_{A}\fr{m}_{A}^{*}\ar[d]^(0,42){j\mu}\ar[r]^(0,43){{\id}\ot p} & \fr{m}_{A}\ot_{A}\fr{m}_{A}^{*}/A\ar[d]^(0.42){0}\ar[r] & 0 \\
0\ar[r] & A\ar[r]^(0,45){j} & \fr{m}_{A}^{*}\ar[r]^(0,43){p} & \fr{m}_{A}^{*}/A\ar[r] & 0 
}
\end{equation}
gives the 6 last terms
\begin{equation}\label{eq.6}
0\ra\tor{A}{1}{\fr{m}_{A}^{*}}{k}\xra{p_{*}}\tor{A}{1}{\fr{m}_{A}^{*}/A}{k}\xra{\delta}
k\ra \fr{m}_{A}^{*}/\fr{m}_{A}{\cdot}\fr{m}_{A}^{*}\ra \fr{m}_{A}^{*}/A\ra 0
\end{equation}
in the long-exact sequence derived from \(-\ot_{A}k\) applied to \(0\ra A\ra \fr{m}_{A}^{*}\ra \fr{m}_{A}^{*}/A\ra 0\). We get that \(\delta=0\) iff the image of \(\mu\) equals \(\fr{m}_{A}\) and \(\delta\) is surjective iff \(\im\mu=A\). By (ii), \(\fr{m}_{A}^{*}/A\cong k\) and the result follows.
\end{proof}
\begin{rem}\label{rem.Finn1}
The short exact sequence of \(R\)-modules \(0\ra R\ra I^{*}\ra S\ra 0\) in Theorem \ref{prop.Finn} is an example of the right \(\cat{X}_{h}\)-approximation given in Theorem \ref{thm.approx3}. The left \(\Pf\)-approximation is given as follows. Lift generators of \(\fr{m}_{A}\) to \(I\) to define a surjective map \(F\ra I\) from a finite \(R\)-free module \(F\). Dualising the short exact sequence \(0\ra \syz{R}{}I\ra F\ra I\ra 0\) gives a short exact sequence \(0\ra I^{*}\ra F^{*}\ra (\syz{R}{}I)^{*}\ra 0\) since \(\xt{1}{R}{I}{R}=0\) by Theorem \ref{prop.Finn} (i) and Proposition \ref{prop.main} (i). The cokernel of the composition \(R\ra I^{*}\ra F^{*}\) defines \(L'\) giving the left \(\Pf\)-approximation of \(S\). Both approximations are contained in the following commutative diagram with short exact rows and columns and the (co)cartesian boxed square:
\begin{equation}\label{eq.box}
\xymatrix@C-4pt@R-12pt@H-30pt{
&& 0\ar[d] & 0\ar[d] \\
0\ar[r] & R \ar[r]\ar@{=}[d] & I^{*} \ar[r]\ar[d]\ar@{}[dr]|{\Box} & S \ar[r]\ar[d] & 0  \\   
0\ar[r] & R \ar[r] & F^{*} \ar[r]\ar[d] & L' \ar[r]\ar[d] & 0  \\
&& (\syz{R}{}I)^{*} \ar@{=}[r]\ar[d] & (\syz{R}{}I)^{*}\ar[d] \\
&& 0 & 0
}
\end{equation}
\end{rem}
\begin{cor}\label{cor.Finn}
With notation and assumptions as in \textup{Theorem \ref{prop.Finn}} except for the last sentence\textup{.} Assume that \(A\) is essentially of finite type over \(k\) and that \(\Spec A\ot_{k}\bar{k}\) only has complete intersection singularities of dimension \(1\) where \(\bar{k}\) is an algebraic closure of \(k\)\textup{.} 
Then the conclusion in \textup{Theorem \ref{prop.Finn}} holds\textup{.}
\end{cor}
\begin{proof}
By \cite[19.3.4]{EGAIV4}, \(A\) is a complete intersection and in particular a \(1\)-dimensional Gorenstein ring. The result hence follows from Theorem \ref{prop.Finn}.
\end{proof}
\begin{rem}\label{rem.Finn2}
In Knudsen's lemma \cite[2.2]{knu:83a}, \(\Spec A\ot_{k}\bar{k}\) is assumed to have only ordinary double points. Knudsen has given a different proof of Corollary \ref{cor.Finn} (i)-(iii) in this case; see \cite{knu:12}. Stabilisation at a point inserts \(\Proj(\fr{m}_{A}^{*}\ot_{A}k)\) (disregarding the case where two points coincides) which by (iv) is \(\BB{P}^{1}_{k}\) if the point is singular and just equals the point if it is regular. See Proposition \ref{prop.mf2}. 

In \cite[3.7]{knu:83a} Knudsen shows that the clutching map is a closed immersion of stacks. In the proof he claims that the image \(I{\cdot}I^{*}\) of the pairing \(I\ot_{R}I^{*}\ra R\) equals \(I\) which implies that \(I^{*}\ot_{R}S\cong S^{\oplus2}\). Both claims are wrong if the family is smoothing or the point moves away from the singularity as the dichotomy in Theorem \ref{prop.Finn} (iv) shows. Knudsen refers to the (completion of) the example in his appendix which we now consider.
\end{rem}
\begin{ex}\label{ex.Finn2}
Let \(S\) be a noetherian ring and put \(R=S[x,y]/(xy-bc)\) for some \(b,c\in S\). The ring homomorphism \(R\ra S\) is given by \(x\mapsto b\) and \(y\mapsto c\) and \(I=(x-b,y-c)\). Put \(\tilde{\vare}=(x-c)\cdot(x+y-b-c)^{-1}\). Calculate \(\tilde{\vare}\cdot(x-b)=x\) and \(\tilde{\vare}\cdot(y-c)=-c\). Hence the multiplication map \(m_{\tilde{\vare}}\) is contained in \(I^{*}\), but not in the submodule \(\nd{}{R}{I}\) unless \(b=0=c\). In particular one obtains \(R/I{\cdot}I^{*}\cong S/(b,c)\) and \(I^{*}\ot_{R}S\cong S{\oplus}S/(b,c)\); (consider \eqref{eq.6} for \(S\ra R\) and \(I\)).

However, in the proof of \cite[3.7]{knu:83a} Knudsen only needs his claims after restricting to the locus of singular curves. For any stable curve \(C\ra T\) with a section \(\Delta:T\ra C\) he defines the singular locus \(C^{\text{sing}}\hra C\) by the ideal sheaf given as the image of a pairing \(\Omega_{C/T}\ot\,\omega_{C/T}^{*}\ra \Q_{C}\). He then only considers the induced stable curve over \(T'=T{\times}_{C}C^{\text{sing}}\). In our (affine) case with \(\omega_{R/S}\cong R\) the pairing becomes the map \(\Omega_{R/S}\ra R\) where \(G\mr{dx}+H\mr{dy}\mapsto Gx-Hy\). The image is the ideal \((x,y)\sbeq R\). The image of \((x,y)\) in \(S\) defines the quotient \(S'=S\ot_{R}R/(x,y)\cong S/(b,c)\). Put \(R'=R\ot_{S}S'\) and \(I'=I\ot_{S}S'\). 
Since the pairing \(I\ot_{R}I^{*}\ra R\) commutes with base change by Theorem \ref{prop.Finn} (i) and Proposition \ref{prop.main} (ii), the image of the pairing \(I'\ot_{R'}(I')^{*}\ra R'\) is \(I'\) and so \(R'/I'{\cdot}(I')^{*}\cong S'\) and \((I')^{*}\ot_{R'}S'\cong (S')^{\oplus 2}\). In the case \(S=k[b,c]\) the henselisation of the corresponding family is versal; see Corollary \ref{cor.main2}, and then this is essentially what is needed in the proof of \cite[3.7]{knu:83a}.
\end{ex}
Fix a noetherian local ring \(\vL\) with residue field \(k\). Let \({}_{\vL}\cat{H}\) be the category of noetherian, henselian, local \(\vL\)-algebras \(S\) with residue field \(k\) and let \({}_{\vL}\cat{H}/k\) be the corresponding comma category of algebras in \({}_{\vL}\cat{H}\) above (i.e.\ with an algebra homomorphism to) \(k\). Let \({}_{\vL}\cat{A}/k\) be the subcategory of artin rings in \({}_{\vL}\cat{H}/k\).
\begin{defn}\label{def.versal}
Let \(F\) and \(G\) be set-valued functors on \({}_{\vL}\cat{H}/k\) (or \({}_{\vL}\cat{A}/k\)) with \(\#F(k)=1=\#G(k)\). A map \(\phi:F\ra G\)
is \emph{smooth} (formally smooth) if the natural map of sets \(f_{\phi}:F(S)\ra F(S_{0})\times_{G(S_{0})}G(S)\) is surjective for all surjections \(\pi:S\ra S_{0}\) in \({}_{\vL}\cat{H}/k\) (respectively \({}_{\vL}\cat{A}/k\)). An element \(\nu\in F(R)\) is \emph{versal} if the induced map \(\hm{}{{}_{\vL}\cat{H}/k}{R}{-}\ra F\) is smooth and \(R\) is algebraic as \(\vL\)-algebra (i.e.\ \(R\) is the henselisation of a finite type \(\vL\)-algebra). An element \(\nu\in F(R)\) is \emph{formally versal} if the induced map \(\hm{}{{}_{\vL}\cat{H}/k}{R}{-}\ra F\) of functors restricted to \({}_{\vL}\cat{A}/k\) is formally smooth.
See \cite{art:74}. 
\end{defn}
Let \({}_{\vL}\cat{HFP}\) be the category with objects \emph{algebraic, flat and pointed algebras} \(\xi\co S\ra T\ra S\) in \({}_{\vL}\cat{H}\), in particular the composition is the identity. The arrows \(\xi_{1}\ra\xi_{0}\) are corresponding commutative diagrams of ring homomorphisms
\begin{equation}\label{eq.fp}
\xymatrix@-0pt@C-0pt@R-8pt@H-0pt{
S_{1} \ar[r]\ar[d]^(0.45){p} & T_{1} \ar[r]\ar[d]^(0.45){q} & S_{1} \ar[d]^(0.45){p} \\
S_{0} \ar[r] & T_{0} \ar[r] & S_{0}
}
\end{equation}
such that the left square is cocartesian (then the right square is cocartesian too). Note that base change exists for the forgetful functor \({}_{\vL}\cat{HFP}\ra {}_{\vL}\cat{H}\). It is given by the henselisation of the tensor product \(R=R_{1}\ot_{S_{1}}S_{2}\) in the maximal ideal \(\fr{m}_{R_{1}}R+\fr{m}_{S_{2}}R\), denoted \(R_{1}\hot_{S_{1}}S_{2}\). The map \(R_{1}\hot_{S_{1}}S_{2}\ra S_{2}\) is the natural one. Hence \({}_{\vL}\cat{HFP}\) is fibred in goupoids above \({}_{\vL}\cat{H}\).

Fix an object \(\xi\co k\ra A\ra k\) in \({}_{\vL}\cat{HFP}\). Let \(\cdf{}{A\ra k}\) denote the corresponding comma category \({}_{\vL}\cat{HFP}/\xi\) of flat and pointed algebras above \(\xi\), called \emph{pointed deformations of \(A\)}. Similarly there is a category \(\cdf{}{A}\) of (unpointed) deformations of \(A\). Both categories are fibred in groupoids above \({}_{\vL}\cat{H}/k\) and there is a map of fibred categories \(\cdf{}{A\ra k}\ra\cdf{}{A}\) by forgetting the pointing. There are corresponding functors \(\df{}{A\ra k}\) and \(\df{}{A}\) from \({}_{\vL}\cat{H}/k\) to \(\Sets\) obtained by identifying all isomorphic objects in the fibres and identifying arrows accordingly. There is an induced map \(\df{}{A\ra k}\ra\df{}{A}\) of functors. We often abuse the notation by hiding the maps to the base object.
\begin{thm}\label{thm.main}
Given an object \(k\ra A\ra k\) in \({}_{\vL}\cat{HFP}\) and assume \(\iota\co S\ra R\) is an unpointed deformation of \(A\)\textup{.} Let \(\id\hot 1\co R\ra R^{(2)}\) be the base change of \(\iota\) by \(\iota\) and \(R^{(2)}\ra R\) the multiplication map\textup{.} Then \(\xi_{v}\co R\ra R^{(2)}\ra R\) is a pointed deformation of \(A\)\textup{.}

If \(\iota\co S\ra R\) gives a formally versal \textup{(}respectively versal\textup{)} element in \(\df{}{A}\) then \(R\ra R^{(2)}\ra R\) gives a formally versal \textup{(}respectively versal\textup{)} element in \(\df{}{A\ra k}\)\,\textup{.}
\end{thm}
\begin{proof}
The residue field of \(R\) is \(k\) and \((R^{(2)}\ra R)\ot_{R}k\cong (A\ra k)\). Hence \(\xi_{v}\) induces an object in \(\cdf{}{A\ra k}\)\,. To prove (formal) versality we consider a map \(\xi_{1}\ra \xi_{0}\) in \(\cdf{}{R\ra k}\) as in \eqref{eq.fp} with surjective vertical maps. Given a map \(\alpha_{0}\co R\ra S_{0}\) such that \(\xi_{v}\hot_{R} S_{0}\cong \xi_{0}\) we show that there is a lifting \(\alpha_{1}\co R\ra S_{1}\) of \(\alpha_{0}\) inducing \(\xi_{1}\).
I.e.\ we consider the following lifting diagram:
\begin{equation}\label{eq.lift}
\xymatrix@-0pt@C+6pt@R+2pt@H-0pt{
& S_{1}\ar[rr]^{\iota_{1}}\ar[d]^{p} && T_{1}\ar[rr]^{\pi_{1}}\ar[d]^{q} && S_{1}\ar[d]_{p}  & \\
& S_{0}\ar[rr]^{\iota_{0}}|!{[ddr];[urr]}\hole && T_{0}\ar[rr]^{\pi_{0}} 
&& S_{0}  & \\
&& R\ar[ul]^{\alpha_{0}}\ar@{-->}[uul]_(0.3){\alpha_{1}}|!{[ul];[ur]}\hole
\ar[rr]^{\id\hot 1}|!{[d];[uur]}\hole|!{[d];[ur]}\hole && R^{(2)}\ar[ul]^(0.55){\beta_{0}}\ar@{-->}[uul]_(0.3){\beta_{1}}|!{[ul];[ur]}\hole
\ar[rr]^(0.45){\mu} && R\ar[ul]^(0.55){\alpha_{0}}\ar@{-->}[uul]_{\alpha_{1}} \\
S\ar@{.>}[uuur]^{\theta_{1}}\ar[uur]_{\theta_{0}}\ar[urr]^{\iota}\ar[rr]^{\iota} && R\ar@{.>}[uuur]^{\tau_{1}}\ar[uur]_(0.3){\tau_{0}}\ar[urr]_{1\hot\id} 
}
\end{equation}
In particular the diagram of solid arrows is commutative with cocartesian squares. We have that \(\iota\in\df{}{A}(S)\) maps to \(\id\hot 1\in\df{}{A}(R)\). Since \(\xi_{v}\) and \(\xi_{0}\) by \(\df{}{A\ra k}\ra\df{}{A}\) maps to \(\id\hot 1\) and \(\iota_{0}\) respectively, it follows that \(\iota\) maps to \(\iota_{0}\).
By versality of \(\iota\) there exists a lifting \(\theta_{1}\) of \(\theta_{0}\) such that \(\iota\) maps to \(\iota_{1}\) in \(\df{}{A}(S_{1})\). The obtained map \(\tau_{1}\) lifts \(\tau_{0}\). Define \(\alpha_{1}\) as \(\pi_{1}\tau_{1}\). Then \(\alpha_{1}\) lifts \(\alpha_{0}\) since \(p\alpha_{1}=p\pi_{1}\tau_{1}=\pi_{0}q\tau_{1}=\pi_{0}\tau_{0}=\alpha_{0}\mu(1\hot\id)=\alpha_{0}\). Also \(S_{1}\hot_{R}R^{(2)}\cong S_{1}\hot_{S}R\cong T_{1}\). Let \(\beta_{1}\) be the induced map \(R^{(2)}\ra T_{1}\). It lifts \(\beta_{0}\) and we get that \(\id\hot 1\) maps to \(\iota_{1}\). But also the right square commutes since \((\pi_{1}\beta_{1}-\alpha_{1}\mu)(1\hot\id)=\pi_{1}\tau_{1}-\alpha_{1}=0\) and \((\pi_{1}\beta_{1}-\alpha_{1}\mu)(\id\hot 1)=\pi_{1}\iota_{1}\alpha_{1}-\alpha_{1}=0\).
\end{proof}
\begin{lem}\label{lem.AQhens}
Let \(h^{\textnormal{ft}}\co S\ra R^{\textnormal{ft}}\) be a finite type homomorphism of noetherian rings\textup{.} Let \(M\) be an \(R^{\textnormal{ft}}\)-module\textup{.} Let \(R\) denote the henselisation of \(R^{\textnormal{ft}}\) in a maximal ideal \(\fr{m}\)\textup{.} 
\begin{enumerate}
\item[(a)] There are natural isomorphisms of Andr{\'e}-Quillen cohomology
\begin{equation*}
\cH^{i}(S,R,R\ot M)\cong\cH^{i}(S, R^{\textnormal{ft}},M)\ot_{R^{\textnormal{ft}}}R\,\,\text{ for all } i\,.
\end{equation*}
\end{enumerate}
Suppose in addition that \(M\) is finite\textup{,} \(h^{\textnormal{ft}}\) is flat\textup{,} \(S\) local henselian and \(S/\fr{m}_{S}\cong R^{\textnormal{ft}}/\fr{m}\cong k\)\textup{.} 
Let \(k\ra A^{\textnormal{ft}}\) denote the central fibre of \(h^{\textnormal{ft}}\) and put \(\fr{m}_{0}=\fr{m}A^{\textnormal{ft}}\)\textup{.} Assume \(\Spec A^{\textnormal{ft}} \setminus\{\fr{m}_{0}\}\) is smooth over \(k\)\textup{.}
\begin{enumerate}
\item[(b)] For all \(i>0\) the Andr{\'e}-Quillen cohomology \(\cH^{i}(S,R,R\ot M)\) is finite as \(S\)-module and there is a natural \(R^{\textnormal{ft}}_{\fr{m}}\)-isomorphism
\begin{equation*}
\cH^{i}(S,R,R\ot M)\cong\cH^{i}(S, R^{\textnormal{ft}},M)_{\fr{m}}\,.
\end{equation*}
\end{enumerate}
\end{lem}
\begin{proof}
See the proof of Lemma 10.1 in \cite{ile:11x}.
\end{proof}
We will use the following notation and assumptions. Let \(x\) denote a sequence of variables \(x_{1},\dots,x_{m}\) and \(f\) a regular sequence of elements \(f_{1},\dots,f_{c}\) contained in \((x)^{2}\subset k[x]\) for a field \(k\). Put \(A^{\textnormal{ft}}=k[x]/(f)\) and assume the Andr{\'e}-Quillen cohomology \(\cH^{1}(k,A^{\textnormal{ft}},A^{\textnormal{ft}}) \cong (A^{\textnormal{ft}})^{\oplus c}/(\im\nabla(f))\) has support in \((x)\) where \(\nabla(f)\) equals the matrix \((\partial f_{i}/\partial x_{j})_{i,j}\). Let \(e_{1},\dots,e_{c}\) be the standard generators in \(k[x]^{\oplus c}\) and pick elements \(g_{1},\dots,g_{N}\) in \(k[x]^{\oplus c}\) such that \(\{e_{1},\dots,e_{c},g_{1},\dots,g_{N}\}\) induce a \(k\)-basis for the finite dimensional \((A^{\textnormal{ft}})^{\oplus c}/(\im\nabla(f))\). Let \(g_{j}^{(1)},\dots,g_{j}^{(c)}\) denote the \(c\) projections of \(g_{j}\) in \(k[x]\). 
Let \(z=z_{1},\dots,z_{c}\) and \(t=t_{1},\dots,t_{N}\) be new sets of variables and define \(F_{i}(x,t)=f_{i}(x)+\sum_{j} t_{j}g_{j}^{(i)}\) in \(k[x,t]\). 
Pick liftings in \(\vL\) of the coeffiecients of \(F_{i}\) to obtain a lifting \(\tilde{F}_{i}\) in \(\vL[x,t]\) of \(F_{i}\). Put \({}^{v\!}\tilde{F}_{i}=\tilde{F}_{i}+z_{i}\in\vL[x,t,z]\). Finally, let \(s=s_{1},\dots,s_{m}\) be another set of variables and let \(\tilde{F}(x,t)-\tilde{F}(s,t)\) denote the sequence \(\tilde{F}_{1}(x,t)-\tilde{F}_{1}(s,t),\dots,\tilde{F}_{c}(x,t)-\tilde{F}_{c}(s,t)\) in \(\vL[x,s,t]\).
\begin{cor}\label{cor.main2}
Let \(A\) be the henselisation of \(A^{\textnormal{ft}}=k[x]/(f)\) in the maximal ideal \((x)A^{\textnormal{ft}}\)\textup{.} Let \(R\ra T\ra R\) be the henselisation of 
\begin{equation*}
\vL[s,t]\ra\vL[x,s,t]/(\tilde{F}(x,t)-\tilde{F}(s,t))\xra{x\mapsto s} \vL[s,t]\,.
\end{equation*}
Then \(R\ra T\ra R\) gives a formally versal element for the pointed deformation functor \(\df{}{A\ra k}\co {}_{\vL}\cat{H}/k\ra \Sets\)\textup{.} If  \(\vL\) is an excellent ring this element is versal\textup{.}
\end{cor}
\begin{proof}
Let \(\iota\co S\ra R\) be the henselisation of \(\vL[t,z]\ra \vL[x,t,z]/({}^{v\!}\tilde{F})\) which we claim gives a (formally) versal element for \(\df{}{A}\). By Theorem \ref{thm.main} \(R\ra R\hot_{S}R\ra R\) gives a (formally) versal element for \(\df{}{A\ra k}\). Note that \(\vL[x,t,z]/({}^{v\!}\tilde{F})\cong\vL[x,t]\) where \(z_{i}\mapsto -\tilde{F}_{i}\). Consider the copy \(\vL[t,z]\ra\vL[s,t]\) where \(z_{i}\mapsto -\tilde{F}_{i}(s,t)\). There is a ring isomorphism \(\vL[s,t]\ot_{\vL[t,z]}\vL[x,t]\cong\vL[x,s,t]/(\tilde{F}(x,t)-\tilde{F}(s,t))\) and the corollary follows. 

The claim is basically well known. We only sketch the argument. By Corollary \ref{cor.nak}, \(({}^{v}\tilde{F})\) is a regular sequence and \(R\) is \(S\)-flat. So \(\iota\) is a deformation of \(A\). To show formal versality put \(\cH^{1}=\cH^{1}(k,A,A)\) which is isomorphic to \(\cH^{1}(k,A^{\text{ft}},A^{\text{ft}})\) by Lemma \ref{lem.AQhens}. Note that \(S_{(1)}:=S/(\fr{m}_{S}^{2}+\fr{m}_{\vL}S)\cong k{\oplus}(\cH^{1})^{*}\). Now \(\iota\) induces the universal deformation in \(\df{}{A}(S_{(1)})\cong \nd{}{k}{\cH^{1}}\) corresponding to the identity. 
Given a lifting situation of elements in \(\cdf{}{A}\) 
\begin{equation}\label{eq.lift2}
\xymatrix@-0pt@C+6pt@R-3pt@H-0pt{
& S^{1}\ar[rr]^{\iota^{1}}\ar[d]^(0.45){p} && R^{1}\ar[d]^(0.45){q} \\
& S^{0}\ar[rr]^{\iota^{0}}|!{[dr];[urr]}\hole && R^{0}  \\
S\ar@{.>}[uur]^{\theta^{1}}\ar[ur]_(0.55){\theta^{0}}\ar[rr]^(0.55){\iota} && R\ar@{.>}[uur]^(0.7){\tau^{1}}\ar[ur]_{\tau^{0}} 
}
\end{equation}
where the vertical maps are surjections, the solid squares are cocartesian and the \(S^{i}\) have finite length.  We need to prove that a lifting \(\theta^{1}\) exists such that the third square is cocartesian and lifts the bottom square. By induction we assume \(\fr{m}_{S^{1}}{\cdot}\ker p=0\). By picking elements \(\theta(z_{i})\) and \(\theta(t_{j})\) in \(S^{1}\) lifting \(\theta^{0}(z_{i})\) and \(\theta^{0}(t_{j})\) we obtain a map \(\theta\co S\ra S^{1}\). Let \(R^{2}=R\hot_{S}S^{1}\). Then \(\iota^{2}\co S^{1}\ra R^{2}\) is in \(\cdf{}{A}\) and lifts \(\iota^{0}\). By obstruction theory there is a transitive action of \(\cH^{1}(S^{0},R^{0},R^{0}\ot\ker p)\) on the set of equivalence classes of liftings of \(\iota^{0}\) to \(S^{1}\) \cite[2.1.3.3]{ill:71}. We have \(R^{0}\ot\ker p\cong A\ot_{k}\ker p\). By \cite[IV 54]{and:74} we get
\begin{equation}
 \cH^{1}(S^{0},R^{0},R^{0}\ot\ker p)\cong \cH^{1}(k,A,A)\ot_{k}\ker p\cong \hm{}{}{(\cH^{1})^{*}}{\ker p}\,.
\end{equation}
Elements here can be `added' to the ring homomorphism \(\theta\) and adjusting by the difference of \(\iota^{1}\) and \(\iota^{2}\) gives a new ring homomorphism \(\theta^{1}\) inducing \(\iota^{1}\) from \(\iota\).

Since \(A^{\text{ft}}\) has an isolated singularity, \cite[Th{\'e}or{\`e}me 8]{elk:73} gives that \(\iota\) is versal. Here we use Artin's Approximation Theorem for an arbitrary excellent coefficient ring \cite{art:69, pop:86,pop:90}; cf.\  R.\ Elkik's remark at the beginning of section 4 in \cite{elk:73}.
\end{proof}
For plane curve singularities the following explicit description of a stably reflexive complex \(E\) which is a hull for the ideal \(I\) in Theorem \ref{prop.Finn} is used to prove that the obvious generalisation of Knudsen's stabilisation \cite{knu:83a} has the relevant features.
\begin{lem}\label{lem.mf2}
Let \(S\ra P\) be a flat homomorphism of local noetherian rings and let \(k=S/\fr{m}_{S}\)\textup{.} Let \(P_{0}\) denote \(P\ot_{S}k\) which we assume is a regular ring of dimension \(2\)\textup{.}  Given an element \(F\in \fr{m}_{P}^{2}\) with image \(f\neq 0\) in \(P_{0}\)\textup{.} Put \(R=P/(F)\) and suppose there is an \(S\)-algebra map \(R\ra S\)\textup{.} Let \(I_{R}\) denote the kernel\textup{.}
\begin{enumerate}
\item[(i)] There are elements \(X_{1},X_{2},G_{1},G_{2}\) in \(\fr{m}_{P}\) with \(X_{1}\) and \(X_{2}\) inducing generators for \(I_{R}\) and satisfying \(X_{1}G_{1}+X_{2}G_{2}=F\)\textup{.}

\item[(ii)] Given such elements\textup{,} 
the \textup{(2 x 2)}-matrices 
\begin{equation*}
\Phi=
\begin{bmatrix}
X_{2} & G_{1} \\
-X_{1} & G_{2}
\end{bmatrix}
\quad\text{and}\quad
\Psi=
\begin{bmatrix}
G_{2} & -G_{1} \\
X_{1} & X_{2}
\end{bmatrix}
\end{equation*}
\end{enumerate}
induce a 2-periodic \(R\)-complex 
\begin{equation*}
\quad\dots\xra{\bar{\Psi}} R^{\oplus 2}\xra{\bar{\Phi}} R^{\oplus 2}\xra{\bar{\Psi}}R^{\oplus 2}\xra{\bar{\Phi}} R^{\oplus 2}\xra{\bar{\Psi}}\dots
\end{equation*}
which is stably reflexive with respect to \(h\) and a hull for \(I_{R}\)\textup{.}
\end{lem}
\begin{proof}
(i) By Corollary \ref{cor.nak} \(R\) is \(S\)-flat.
Let \(I_{P}\) be the kernel of the induced \(S\)-algebra map \(P\ra S\). Then \(I_{P}\ot_{S}k\cong\fr{m}_{P_{0}}\) which is generated by two elements, say \(x_{1}\) and \(x_{2}\). Pick liftings \(X_{i}\) in \(I_{P}\) of the \(x_{i}\). Then \(I_{P}\) is generated as ideal by \(X_{1}\) and \(X_{2}\). The kernel of \(I_{P}\ra I_{R}\) equals the kernel \((F)\) of \(P\ra R\). In particular \(F=X_{1}G_{1}+X_{2}G_{2}\) for some elements \(G_{1}\) and \(G_{2}\) which have to be non-units. 

(ii) Consider the commutative diagram with exact rows
\begin{equation}\label{eq.cok}
\xymatrix@-0pt@C+12pt@R-6pt@H-0pt{
0\ar[r] & P\ar[d]_(0.4){\tau}\ar[r]^(0.45){[X_{2},-X_{1}]^{\text{tr}}} & P^{\oplus 2}\ar[r]^{[X_{1},X_{2}]}\ar@{=}[d] & I_{P}\ar[r]\ar[d]^(0.42){\rho} & 0 \\
0\ar[r] & P^{\oplus 2}\ar[r]^{\Phi} & P^{\oplus 2}\ar[r] & \coker \Phi\ar[r] & 0
}
\end{equation}
The (inverse) connecting isomorphism \(P=\coker\tau\cong\ker\rho\) takes \(1\) to \(F\in I_{P}\). We get \(\coker\Phi\cong I_{P}/(F)\cong I_{R}\).
The pair \((\Phi,\Psi)\) is a matrix factorisation of \(F\). Then \(C(\Phi,\Psi)\) is a stably reflexive complex and a hull for \(I_{R}\) by Corollary \ref{cor.mf}.
\end{proof}
\begin{prop}\label{prop.mf2}
In addition to the assumptions in \textup{Lemma \ref{lem.mf2}} assume \(P= S[X_{1},X_{2}]_{\fr{m}}\) where \(\fr{m}\) is the maximal ideal \(\fr{m}_{S}+(X_{1},X_{2})\)\textup{.} Put \(A=P_{0}/(f)\) and let \(x_{i}\) and \(g_{i}\) denote the images in \(P_{0}\) of \(X_{i}\) and \(G_{i}\) respectively for \(i=1,2\)\textup{.} Suppose the \(g_{i}\) are contained in \(k[x_{1},x_{2}]\)\textup{.}
\begin{enumerate}
\item[(i)] The map \(\pi^{\textnormal{s}}\co \Proj\Sym_{R} (I_{R}^{*})\ra \Spec S\) is flat\textup{.}
\item[(ii)] The closed fibre \(\pi^{\textnormal{s}}{\times}_{S}\Spec k\) is isomorphic to \(\pi_{0}^{\tn{s}}\co \Proj\Sym_{A}(\fr{m}_{A}^{*})\linebreak[1]\ra\Spec k\) and has local complete intersection singularities\textup{.}
\item[(iii)] The exceptional fibre \(E\) of \(\Proj\Sym_{A}(\fr{m}_{A}^{*})\ra\Spec A\) is isomorphic to \(\BB{P}^{1}_{k}\)\textup{.}
\item[(iv)] The quotient \(I_{R}^{*}\ra I_{R}^{*}/R\) induces a section of \(\pi^{\textnormal{s}}\) which in the closed fibre gives a smooth point contained in \(E\)\textup{.}
\end{enumerate}
\end{prop}
\begin{proof}
(i) By Lemma \ref{lem.mf2}, \(I_{R}^{*}\cong\coker\Phi^{\text{tr}}\) , hence
\begin{equation}\label{eq.hom}
\Sym_{R}(I_{R}^{*})\cong R[U,V]/(X_{2}U+G_{1}V,G_{2}V-X_{1}U)\,.
\end{equation}
Put \(v=V/U\). Then \(\Q_{D_{+}(U)}\cong P[v]/(X_{2}+G_{1}v, X_{1}-G_{2}v)\) as \(F\) is contained in the ideal. We claim that \((x_{2}+g_{1}v,x_{1}-g_{2}v)\) is a regular sequence ideal. It follows that \((X_{2}+G_{1}v, X_{1}-G_{2}v)\) is a regular sequence ideal and that \(\Q_{D_{+}(U)}\) is \(S\)-flat by Corollary \ref{cor.nak}. As the claim is independent of the choice of presentation \(\Phi\ot_{S}k\) we can assume that \(g_{2}\) is contained in \(k[x_{2}]\). Then \(\Q_{D_{+}(U)}\ot_{S}k\) is isomorphic to \((P_{0}/(x_{1}))[v]/(x_{2}+vg'_{1})\) where \(g'_{1}=g_{1}(g_{2}v,x_{2})\) and the claim follows. Put \(u=U/V\). Then \(\Q_{D_{+}(V)}\cong P[u]/(X_{2}u+G_{1},X_{1}u-G_{2})\). As for the \(U\)-chart it is sufficient to show that \((x_{2}u+g_{1},x_{1}u-g_{2})\) is a regular sequence ideal. We can assume that \(g_{2}\) is contained in \(k[x_{2}]\). For \(g_{2}\neq0\) one checks that \(x_{1}u-g_{2}\) is irreducible in \(k[x_{1},x_{2},u]\). If \(g_{2}=0\) a direct argument shows that \((x_{2}u+g_{1},x_{1}u)\) is a regular sequence.

(ii) The closed fibre of \(\pi^{\text{s}}\) is obtained from \(I_{R}^{*}\ot_{S}k\). By Theorem \ref{prop.Finn} (i) and Proposition \ref{prop.main} (iii), \(I_{R}^{*}\ot_{S}k\cong \fr{m}_{A}^{*}\). By the argument for (i) the closed fibre has local complete intersection singularities. 

(iii) The exceptional fibre \(E\) is given by \(\Proj\Sym_{A}(\fr{m}_{A}^{*}\ot_{A}k)\). By Theorem \ref{prop.Finn} (iv), \(\fr{m}_{A}^{*}\ot_{A} k\cong k^{\oplus 2}\) as \(A\)-modules. 

(iv) We have \(I_{R}^{*}/R\cong S\) by Theorem \ref{prop.Finn} (ii) which gives a section of \(\pi^{\tn{s}}\) by \cite[4.2.3]{EGAII}. Inspecting the presentation in Lemma \ref{lem.mf2} we find that the image of \(1\) under the map \(R\ra I_{R}^{*}\) is given by \([X_{1},X_{2}]^{\text{tr}}\). The section of \(\pi^{\text{s}}\) is therefore obtained by dividing the homogeneous coordinate ring out by \(V\) which gives the quotient ring \(R[U]/(U)(X_{1},X_{2})\). 
\end{proof}
\section{The stabilisation map}\label{sec.stab}
Let \(\bar{\ms{M}}_{g,n}\) denote the stack of stable \(n\)-pointed curves. In particular it is a category fibred in groupoids over the base category of schemes; cf.\ Section \ref{sec.approx}. An object in \(\bar{\ms{M}}_{g,n}\) is a proper and flat map \(\pi\co C\ra T\) of schemes together with \(n\) sections \(\sigma_{i}\co T\ra C\) such that the geometric fibres \(C_{\bar{t}}\) of \(\pi\) are connected curves which together with the points \(\sigma_{i}(\bar{t})\) have certain properties; see \cite{knu:83a}. E.g.\ the only singularities on the curve should be ordinary double points (nodes), the \(n\) points should be smooth and distinct, and the \(n\)-pointed curve should have finite automorphism group.
The morphisms in \(\bar{\ms{M}}_{g,n}\) are commutative, cartesian diagrams. There is also a stack \(\bar{\ms{C}}_{g,n}\) with objects stable \(n\)-pointed curves plus an extra section \(\Delta\co T\ra C\) without conditions. Forgetting this extra section gives a morphism of fibred categories \(\bar{\ms{C}}_{g,n}\ra \bar{\ms{M}}_{g,n}\) and \(\bar{\ms{C}}_{g,n}\) is called the universal curve. The stabilisation map is a morphism of fibred categories \(s\co \bar{\ms{C}}_{g,n}\ra \bar{\ms{M}}_{g,n+1}\). Knudsen uses it to study divisors on the \(\bar{\ms{M}}_{g,n}\) and to construct the clutching maps. We sketch the main arguments in the construction of \(s(\pi,\{\sigma_{i}\},\Delta)=(\pi^{s}\co C^{s}\ra T,\{\sigma_{i}^{s}\})\). We assume that all schemes are noetherian.

Since the points of \(\sigma_{i}(T)\) in \(C\) are smooth over \(T\) by \cite[6.7.8]{EGAIV2}, the \(\sigma_{i}(T)\) are locally principal divisors of \(C\) by \cite[17.10.4]{EGAIV4}, but this is not the case for \(\Delta\).  Let \(\mc{I}=\mc{I}_{\Delta(T)}\) be the ideal sheaf in \(\Q_{C}\) defining \(\Delta\). Let \(\Q_{C}\ra \Q_{C}(\Sigma_{i} \sigma_{i})\) and \(\Q_{C}\ra \mc{I}^{*}\) be the duals of the ideal inclusions. Let the coherent sheaf \(\mc{F}=\mc{F}_{C/T}\) on \(C\) be defined by the exact sequence
\begin{equation}\label{eq.K}
0\ra \Q_{C}\lra \mc{I}^{*}\oplus\Q_{C}(\Sigma_{i}\sigma_{i})\xra{\,\,\rho\,\,} \mc{F}_{C/T}\ra 0\,.
\end{equation}
Let \(q\co C^{s}:=\Proj\Sym_{\Q_{C}}\!\mc{F}\ra C\) and \(\pi^{s}=\pi q\). The quotient \(\mc{F}\ra \mc{F}/\rho(\Q_{C}(\Sigma_{i} \sigma_{i}))\linebreak[1]=\mc{L}_{\Delta}\) defines a section \(\Delta^{s}\co T\ra C^{s}\) lifting \(\Delta\) provided we know that \(\Delta^{*}\mc{L}_{\Delta}\) is a line bundle on \(T\); cf.\ \cite[4.2.3]{EGAII}. This local question is answered in Corollary \ref{cor.Finn} (ii). Let \(\sigma_{n+1}^{s}:=\Delta^{s}\). Similarly \(\sigma_{i}^{*}(\mc{F}/\rho(\mc{I}^{*}))\) defines the lifting \(\sigma_{i}^{s}\) for \(i=1,\dots,n\). If \eqref{eq.K} commutes with base change \(s\) will define a functor (after choosing representatives which by construction are unique up to unique isomorphisms). Since there is a global comparison map \(f^{*}(\mc{F}_{C/T})\ra \mc{F}_{C'/T'}\) where \(f\co C'=C{\times}_{T}T'\ra C\) is the projection, the question is local. The critical case is where \(\Delta(\bar{t})\) is an ordinary double point. After localisation \(\mc{F}\) equals \(\mc{I}^{*}\) and \(f^{*}(\mc{I}^{*})\ra (f^{*}\mc{I})^{*}\) is an isomorphism by Corollary \ref{cor.Finn} (i). 

To prove flatness of \(C^{s}\ra T\), we consider the local situation. Put \(t=\pi(x)\) for  \(x\in C\), let \(\Q\) and \(\vL\) denote the henselisations of \(\Q_{C,x}\) and \(\Q_{T,t}\) respectively. Moreover, let \(k=k(t)\) and \(A=\Q\ot_{T}k\) with \(k\)-algebra map \(A\ra k\) given by the section \(\Delta_{C,x}\ot k\). Knudsen's idea is to write up an explicit (formally) versal family for \(\df{}{A\ra k}\), say with base ring \(R\), do the local version of the construction of \(C^{s}\) over this family and inspect it for flatness. By versality (assuming \(\vL\) is excellent) \(\vL\ra\Q\) is obtained by base change along some \(\vL\)-algebra map \(R\ra\vL\). This family is then flat and flatness of \(C^{s}\) along the fibre over \(x\) follows by faithful flatness of henselisation. The non-excellent case is similar, formal versality is sufficient, and one can even work with complete rings and formal families. 

For the explicit formally versal (f.v.) family we consider the critical case where \(\Delta(\bar{t})\) is an ordinary double point. In \cite{knu:12} Knudsen shows that \(\hat{\Q}\ot_{\vL}k(t)\) is isomorphic to \(k(t)[[x_{1},x_{2}]]/(q)\) where \(q=q(x)=x_{1}^{2}+\gamma x_{1}x_{2}+\delta x_{2}^{2}\) for \(\gamma\) and \(\delta\) in \(k=k(t)\) with discriminant \(\gamma^{2}-4\delta\neq 0\). One could instead take the strict henselisation (also faithfully flat) of the local rings \(\Q_{T,t}\) and \(\Q_{C,x}\) and \(q\) would split as a product of two distinct linear forms over \(k^{\text{sep}}\), but this simplification of the equation would only change the following argument nominally. Let \(A\) be the henselisation of \(k[x_{1},x_{2}]/(q)\). By Theorem \ref{thm.main} one obtains the f.v.\ family of \(\df{}{A\ra k}\) from the f.v.\ family of \(\df{}{A}\). The Zariski tangent space \(\df{}{A}(k[\vare])\) equals the first Andr{\'e}-Quillen cohomology \(\cH^{1}(k,A,A) \cong A/(\im\nabla(q))\) where \(\nabla(q)\) is the Jacobi matrix. By the discriminant condition  
\begin{equation}
\df{}{A}(k[\vare])\cong k[x_{1},x_{2}]/(q,2x_{1}+\gamma x_{2},2\delta x_{2}+\gamma x_{1})\cong k. 
\end{equation}
Put \(\tilde{q}=\tilde{q}(x)=x_{1}^{2}+\tilde\gamma x_{1}x_{2} + \tilde{\delta}x_{2}^{2}\) where \(\tilde{\gamma}\) and \(\tilde{\delta}\) in \(\vL\) are liftings of \(\gamma\) and \(\delta\). Put \({}^{v}\tilde{Q}=\tilde{q}+z\) for a variable \(z\) and let \(S\ra R\) be the henselisation of \(\vL[z]\ra\vL[x_{1},x_{2},z]/({}^{v}\tilde{Q})\). This is a f.v.\ family for \(\df{}{A}\) (cf.\ proof of Corollary \ref{cor.main2}). Note that \(\vL[x_{1},x_{2},z]/({}^{v}\tilde{Q})\cong\vL[x_{1},x_{2}]\) where \(z\mapsto -\tilde{q}(x)\). Corollary \ref{cor.main2} shows that the f.v.\ family \(R\ra R^{(2)}\ra R\) in \(\df{}{A\ra k}\) is particularly simple in this case. We use \(s_{1},s_{2}\) instead of the variables \(x_{1},x_{2}\) in the left \(R\).  The obvious ring homomorphism 
\begin{equation}
L:=\vL[s_{1},s_{2},x_{1},x_{2}]/(\tilde{q}(x)-\tilde{q}(s))\lra \vL[s_{1},s_{2}]\ot_{\vL[z]}\vL[x_{1},x_{2}]
\end{equation}
is an isomorphism, \(L^{\text{h}}=R^{(2)}\) and \(\hat{L}\) is precisely Knudsen's hull in Proposition 2.1 in \cite{knu:12}.

The equations for \(C^{s}\) above \(x\in C\) (after completion) is given in Proposition \ref{prop.mf2}. Let \({}^{\text{v}\!}S\) and \({}^{\text{v}\!}P\) denote the Zariski localisation in the obvious maximal ideals of \(\vL[s_{1},s_{2}]\) and \({}^{\text{v}\!}S[x_{1},x_{2}]\), respectively. Put \(F=\tilde{q}(x)-\tilde{q}(s)\) and \({}^{\text{v}\!}R={}^{\text{v}\!}P/(F)\). Let \(I\) be the kernel of the \({}^{\text{v}\!}S\)-algebra map \({}^{\text{v}\!}R\ra {}^{\text{v}\!}S\) defined by \(x_{i}\mapsto s_{i}\) for \(i=1,2\). To find the differential in the stably reflexive complex which is a hull for \(I\) (see Lemma \ref{lem.mf2}), put \(X_{i}=x_{i}-s_{i}\) for \(i=1,2\). Then \(G_{1}=x_{1}+\tilde\gamma x_{2}+s_{1}\) and \(G_{2}=\tilde\delta (x_{2}+s_{2})+\tilde\gamma s_{1}\) gives a solution to the equation \(F=X_{1}G_{1}+X_{2}G_{2}\). The homogeneous coordinate ring is given in \eqref{eq.hom}. In the closed fibre the local charts are (without localisation in \((x)\)):
\begin{align}
\Q_{D_{+}(U)}\ot k &\cong k[x_{2},v]/(x_{2}(\delta v^{2}+\gamma v+1))\label{eq.v} \\
\Q_{D_{+}(V)}\ot k &\cong k[x_{2},u]/(x_{2}(u^{2}+\gamma u+ \delta))
\end{align}
In particular these are local complete intersections and \(\Proj\Sym_{{}^{\text{v}\!}R}I^{*}\) is flat over \({}^{\text{v}\!}S\). Moreover, the closed fibre is reduced and connected, the exceptional component \(x_{1}=0=x_{2}\) is a \(\BB{P}^{1}_{k}\), and the intersection points with the other components in the geometric fibre are two distinct ordinary double points as the discriminants are non-zero. Also note that the image of \(1\) under \({}^{\text{v}\!}R\ra I^{*}\) is given by \([X_{1},X_{2}]^{\text{tr}}\). The section \(\Delta^{s}\), defined by the quotient \(I^{*}/\,{}^{\text{v}\!}R\), is hence given by dividing out by the homogeneous coordinate \(V\). In the local chart \(D_{+}(U)\) this corresponds to \(v=0\) which is never a solution to \(\delta v^{2}+\gamma v+1=0\).

Note that while Corollary \ref{cor.Finn} (i) and (ii) are used in this argument, Corollary \ref{cor.Finn} (iii) is not. We use Theorem \ref{prop.Finn} (iv) in the proof of Proposition \ref{prop.mf2} to show that stabilisation inserts a \(\BB{P}^{1}\) at a singular point.
In Knutsen's proof of \cite[3.7]{knu:83a}, which states that the clutching map is a closed immersion of stacks, he claims that Theorem \ref{prop.Finn} (iv) generalises to families in the case the point on the closed fibre is a node. In Remark \ref{rem.Finn2} and Example \ref{ex.Finn2} we explain why this is wrong and why it does not matter for the proof.
\providecommand{\bysame}{\leavevmode\hbox to3em{\hrulefill}\thinspace}
\providecommand{\MR}{\relax\ifhmode\unskip\space\fi MR }
\providecommand{\MRhref}[2]{%
  \href{http://www.ams.org/mathscinet-getitem?mr=#1}{#2}
}
\providecommand{\href}[2]{#2}

\end{document}